\documentclass{article}
\usepackage{graphicx} 
\usepackage{amsthm}
\usepackage{amsmath,amsfonts,amssymb}
\usepackage{a4wide}
\usepackage[usenames,dvipsnames]{color}
\usepackage[verbose,colorlinks=true,linktocpage=true,linkcolor=blue,citecolor=blue]{hyperref}
\usepackage{footnote}
\usepackage{tikz}
\usepackage{extarrows}
\usepackage{array} 
\usepackage[all, cmtip]{xy} 

\usepackage{hyperref}
\usepackage{marvosym}
\usepackage{mathrsfs} 
\usepackage{indentfirst} 
\usepackage{accents} 

\usepackage{chngcntr}
\counterwithin{table}{section}

\theoremstyle{definition}
\newtheorem{definition}{Definition}[section]

\theoremstyle{plain}
\newtheorem{theorem}[definition]{Theorem}
\newtheorem{proposition}[definition]{Proposition}
\newtheorem{lemma}[definition]{Lemma}
\newtheorem{corollary}[definition]{Corollary}

\newtheorem{conjecture}[definition]{Conjecture}
\theoremstyle{remark}
\newtheorem{remark}[definition]{Remark}

\numberwithin{equation}{section}

\newcommand\keywords[1]{\textbf{keywords}:#1}

\newcommand{\EmailD}[1]{\href{mailto:#1}{\Letter\ #1}}

\hypersetup{
    colorlinks=true, 
    linkcolor=blue, 
    urlcolor=blue, 
}

\allowdisplaybreaks[1]

\title{Drinfeld super Yangian of the exceptional Lie superalgebra $D(2,1;\lambda)$}
\author{Hongda Lin$^{1}$ and Honglian Zhang$^{2,}$\thanks{Corresponding Author.}}
\date{}

\begin{document}

\maketitle

\begin{center}
\footnotesize
\begin{itemize}
\item[$1$] Shenzhen International Center for Mathematics, Southern University of Science and Technology, P.R. China.
\EmailD{linhd@sustech.edu.cn}

\item[$2$] Department of Mathematics, Shanghai University, Shanghai 200444, P.R. China.

\EmailD{hlzhangmath@shu.edu.cn}
\end{itemize}
\end{center}

\begin{abstract}
In this paper, we establish the first rigorous framework for the Drinfeld super Yangian associated with an exceptional Lie superalgebra, which lacks a classical Lie algebraic counterpart. Specifically, we systematically investigate the Drinfeld presentation and structural properties of the super Yangian associated with the exceptional Lie superalgebra $D(2,1;\lambda)$. First, we introduce a Drinfeld presentation for the super Yangian associated with the exceptional Lie superalgebra  $D(2,1;\lambda)$, explicitly constructing its current generators and defining relations. A key innovation is the construction of a Poincaré-Birkhoff-Witt (PBW) basis using degeneration techniques from the corresponding quantum loop superalgebra. Furthermore, we demonstrate that the super Yangian possesses a Hopf superalgebra structure, explicitly providing the coproduct, counit, and antipode.
\end{abstract}

\section{Introduction}
\textit{Quantum groups}, independently introduced by Drinfeld \cite{Dr85} and Jimbo \cite{Ji85} around 1985, have become a fundamental domain in modern algebraic research and mathematical physics. This theory has profoundly influenced various areas, including algebraic geometry, low-dimensional topology, combinatorics, quantum field theory, statistical mechanics, and quantum computation.
In the framework established by Drinfeld and Jimbo, quantum groups are specific families of Hopf algebras that deform the universal enveloping algebras of finite-dimensional simple Lie algebras or Kac-Moody algebras, widely recognized as the \textit{Drinfeld-Jimbo presentation}. This presentation is instrumental in constructing solutions to the quantum Yang-Baxter equations.

Two prominent families of quantum groups are \textit{Yangians} and \textit{quantum affine algebras}, both of which admit a presentation established by Drinfeld \cite{Dr87}, known as the \textit{Drinfeld presentation}. This new presentation has been widely used to study the highest weight representations and classify finite-dimensional irreducible representations \cite{CP95}. Drinfeld provided a conjecture that the isomorphism between this new presentation and the Drinfeld-Jimbo presentation, while the general proofs for quantum affine algebras were given by Beck \cite{Be94}, Damiani \cite{Da12}, and Jing-Zhang \cite{JZ16} using braid group constructions.

\textit{Quantum superalgebras} are defined as supersymmetric generalizations of quantum groups, specifically designed to describe supersymmetric physical fields. Bracken, Gould, and Zhang \cite{BGZ90,ZGB91} developed an R-matrix that solves the quantum Yang-Baxter equation of $\mathbb{Z}_2$-graded. Yamane \cite{Ya94-2,Ya99} systematically established Serre-type presentations for all classical types of (quantum) affine Lie superalgebras and introduced the Drinfeld presentation for quantum affine superalgebras of the special linear Lie superalgebra $\mathfrak{sl}_{m|n}$. Subsequent research by Yamane and collaborators \cite{HSTY08} extended these results to the Lie superalgebra $D(2,1;\lambda)$. Recently, an isomorphism between the Drinfeld-Jimbo presentation and the Drinfeld presentation of the quantum affine superalgebra $U_q(\widehat{\mathfrak{sl}}_{m|n})$ has been established using a minimalistic presentation with a finite set of generators and defining relations \cite{LinYZ24}.

The study of super Yangians was initiated by Nazarov \cite{Na91}, who generalized the construction of the classical Yangian $Y(\mathfrak{gl}_n)$ to superalgebras. The Drinfeld presentation of super Yangians has been developed for several families of Lie superalgebras, including the general linear cases \cite{St94,Go07}, the orthosymplectic cases \cite{FT23,Mo24}, and the queer cases \cite{CW25}. Additionally, the Hopf superalgebra structures of Drinfeld super Yangians of types $A(m,n)$ and $B(m,n)$ have been explicitly described \cite{Ma23,MS22}. However, the super Yangians for exceptional Lie superalgebras remain an open problem, with no systematic results currently available.

It has long been recognized that Yangians can be realized as degenerate limits of quantum loop algebras, which are quotients of quantum affine algebras with trivial central elements. This fact was first stated by Drinfeld \cite{Dr86} and later rigorously proved by \cite{GM12} and \cite{GT13} using different methods. Gautam and Toledano Laredo also established an equivalence of tensor categories that identifies finite-dimensional representations of Yangians with those of quantum loop algebras in the symmetrizable Kac-Moody setting. Following Guay-Ma's \cite{GM12} proof, analogous results have been extended to twisted Yangians \cite{CG15} and super Yangians \cite{LinWZ24}. In other breakthrough results \cite{LWZ25,LZ24}, the authors demonstrated that the Drinfeld presentation for twisted Yangians for (quasi-)split types can be regarded as an appropriate degeneration of $\imath$quantum groups for the same types. 
These findings inspire the development of Drinfeld super Yangians for $D(2,1;\lambda)$ by utilizing the degeneration approach from the corresponding quantum loop superalgebras.

The Lie superalgebras $D(2,1;\lambda)$ \cite{Ka77,Mu12} are parametrized by a one-parameter family of 17-dimensional nonisomorphic simple Lie superalgebras, with the special cases $\lambda\in\{1,-2,-1/2\}$ recovering the orthosymplectic Lie superalgebra $D(2,1)=\mathfrak{osp}_{4|2}$. Like $F(4)$ and $G(3)$, $D(2,1;\lambda)$ is a class of the exceptional basic classical Lie superalgebras. Its peculiarity lies in 
the lack of a Lie algebraic analogue, setting it apart from the general linear, orthosymplectic and other exceptional Lie superalgebras. Moreover, the deep connections of $D(2,1;\lambda)$ with the superconformal algebra, AdS/CFT correspondence and Tits construction \cite{BEl02,DF04,HSTY08,Mal99} is very promising us to study its associated super Yangians.  

In this work, we will introduce the definition of the Drinfeld super Yangian $Y(D(2,1;\lambda))$ whose central is trivial. The aim of this work is to investigate several important structures for $Y(D(2,1;\lambda))$, including: (1)~its PBW basis, (2)~classical limit, (3)
~minimalistic presentation and (4)~coproduct. Our results coincides with the orthosymplectic Yangian $Y(\mathfrak{osp}_{4|2})$ \cite{FT23} when $\lambda$ specializes to $1$, $-2$, or $-1/2$. 

The paper is organized as follows. 
In Section \ref{se:simpleLiesuper}, we review the two non-conjugacy Dynkin diagrams and their corresponding root vectors for the exceptional Lie superalgebra $D(2,1;\lambda)$. We select dual bases compatible with the standard nondegenerate bilinear form and employ them to define the (half) Casimir element for $D(2,1;\lambda)$. Furthermore, we derive a current presentation for the affine Lie superalgebra $\widehat{D}(2,1;\lambda)$ and explain briefly how it degenerates into the associated polynomial Lie superalgebra in a centain recursive way. 
Section \ref{se:quantumAFsuper} is devoted to reformulate the isomorphism---established in \cite{HSTY08}---between the Drinfeld-Jimbo and the Drinfeld presentation of quantum affine superalgebra $U_q(\widehat{D}(2,1;\lambda))$. This reformulation enables us to develop the classical limit of $U_q(\widehat{D}(2,1;\lambda))$ in terms of Drinfeld current generators. 
In Section \ref{se:superYangian}, we first present the Drinfeld presentation of the super Yangian $Y(D(2,1;\lambda))$, establishing its current generators and defining relations. We then demonstrate that this structure can be constructed as the limit form of the quantum loop superalgebra associated with $D(2,1;\lambda)$. Based on this degeneration framework, we construct the PBW basis, which allows us to explicitly characterize its classical limit to the corresponding universal enveloping superalgebra. 
Meanwhile, we conjecture that the odd reflection within the framework of $ Y_{\hbar}(\mathfrak{g}) $  transforms the associated Dynkin diagram, motivated by the isomorphism between the quantum affine superalgebras $ U_q(\widehat{D}_{\lambda}) $ and $ U_q(\widehat{D}_{\lambda}^{\circ}) $ \cite{HSTY08}.
Finally, Section \ref{se:hopfstructure} provides the minimalistic presentation for $Y(D(2,1;\lambda))$, demonstrating that it admits a Hopf superalgebraic structure endowed with an explicit coproduct, counit and antipode.

\medskip
\section{Simple Lie superalgebras $D(2,1;\lambda)$ and its current algebras}\label{se:simpleLiesuper}
Let $\mathbb{C}$, $\mathbb{Z}$, $\mathbb{Z}_+$, $\mathbb{N}$ denote the sets the set of complex numbers, integer numbers, non-negative integer numbers and positive integer numbers, respectively. Throughout this paper, all (super-)spaces, associative (super-)algebras, Lie (super-)algebras, homomorphisms, isomorphisms and (anti-)automorphisms are defined over $\mathbb{C}$ unless otherwise specified. 
Let $\delta_{ab}$ be the Kronecker function equals $1$ if $a=b$ and $0$ otherwise. 
Write $\mathbb{Z}_2=\mathbb{Z}/2\mathbb{Z}=\left\{\bar{0},\bar{1}\right\}$. The $\mathbb{Z}_2$-grading of an element $x$ is denoted by $|x|$. If $|x|=\bar{0}$ for a homogeneous element $x$, we call $x$ an even element; otherwise, it is odd. 

If both $\mathcal{A}=\mathcal{A}_{\bar{0}}\oplus \mathcal{A}_{\bar{1}}$ and $\mathcal{B}=\mathcal{B}_{\bar{0}}\oplus \mathcal{B}_{\bar{1}}$ are associative (Lie) superalgebras, the tensor product $\mathcal{A}\otimes\mathcal{B}$ also forms an associative (Lie) superalgebra with the following multiplication rule: 
$$(x_1\otimes y_1)(x_2\otimes y_2)=(-1)^{|y_1||x_2|}x_1x_2\otimes y_1y_2,$$
for any homogeneous $x_1, x_2\in \mathcal{A}$, $y_1,y_2\in \mathcal{B}$. 
Define
\begin{gather*}
    (\mathcal{A}\otimes \mathcal{B})_{\bar{0}}=(\mathcal{A}_{\bar{0}}\otimes \mathcal{B}_{\bar{0}})\oplus (\mathcal{A}_{\bar{1}}\otimes \mathcal{B}_{\bar{1}}), \quad
    (\mathcal{A}\otimes \mathcal{B})_{\bar{1}}=(\mathcal{A}_{\bar{0}}\otimes \mathcal{B}_{\bar{1}})\oplus (\mathcal{A}_{\bar{1}}\otimes \mathcal{B}_{\bar{0}}).
\end{gather*}
Then $\mathcal{A}\otimes \mathcal{B}$ is a $\mathbb{Z}_2$-graded algebra with the decomposition $$\mathcal{A}\otimes \mathcal{B}=(\mathcal{A}\otimes \mathcal{B})_{\bar{0}}\oplus (\mathcal{A}\otimes \mathcal{B})_{\bar{1}}.$$

\subsection{Non-conjugacy Dynkin diagrams}\label{se:simpleLiesuper:DG}
Let $\mathfrak{g}$ be the simple Lie superalgebra of type $D(2,1;\lambda)$ for $\lambda\in\mathbb{C}\setminus\{0,-1\}$. The exclusion of $0$ and $-1$ is necessary to ensure the simplicity of $\mathfrak{g}$. As stated in \cite[Section 4.2]{Mu12}, $\mathfrak{g}$ has dimension 17 and admits a standard generator set $\{e_i,f_i,h_i|i=1,2,3\}$ with the Cartan subalgebra
$\mathfrak{h}=\operatorname{span}\{h_i|i=1,2,3\}$. The $\mathbb{Z}_2$-grading of the generators is given by $|e_i|=|f_i|$ and $|h_i|=\bar{0}$. For notational convenience, we fix the index set $I=\{1,2,3\}$ and extend the $\mathbb{Z}_2$-grading on $I$ via  $|i|=|e_i|$ for all $i\in I$. 

Recall that \cite[Section 2.5]{Ka77}, $\mathfrak{g}$ possesses two non-conjugacy Dynkin diagrams with associated Cartan matrices, as shown in Table \ref{table:1}. Here, the $i$-th node 
\begin{tikzpicture}
           \draw (0,0) circle (0.20); 
          \end{tikzpicture}
means $|e_i|=\bar{0}$, while 
\begin{tikzpicture}
           \draw (0,0) circle (0.20);
           \draw (-0.15,0.15)--(0.15,-0.15) (0.15,0.15)--(-0.15,-0.15);
           \end{tikzpicture}
           means $|e_i|=\bar{1}$. 
\begin{table}[h]\label{table:1}
    \centering
    ~\\[-0.4em]
    \begin{tabular}{|m{0.6cm}|m{4.5cm}|m{4cm}|}
        \hline
        ~&~&~\\[-0.8em]
       ~~$\mathfrak{g}$ & \hspace{2.4em}Dynkin diagram $G$ & \hspace{2em}Cartan matrix $A$   \\[0.2em]
        \hline
        ~&~&~\\
            ~$D_{\lambda}$ &
        ~\begin{tikzpicture}
           \draw (0,0) circle (0.22) (0.22,0)--(1.64,0);
           \node at (0,0.6) {$1$};
           \node at (0.93,0.25) {$1$}; 
           \draw (1.86,0) circle (0.22);
           \node at (1.86,0.6) {$2$};  
           \draw (1.70,0.16)--(2.02,-0.16) (1.70,-0.16)--(2.02,0.16);
           \draw (3.72,0) circle (0.22)  (2.08,0)--(3.50,0);
           \node at (3.72,0.6) {$3$};
           \node at (2.79,0.25) {$\lambda$}; 
           \end{tikzpicture} & 
            \hspace{2em}$\begin{array}{c}
                 \begin{pmatrix}
            2 & -1 & 0\\
            1 & 0 & \lambda \\
            0 & -1 & 2
        \end{pmatrix}  \\
                 ~ 
            \end{array}$   \\
        \hline
        ~&~&~\\
            ~$D_{\lambda}^{\circ}$  &
        \hspace{2em}\begin{tikzpicture}
          \draw (0,0.8) circle (0.22)  (0,0.58)--(0,-0.58);
          \draw (-0.16,0.96)--(0.16,0.64) (0.16,0.96)--(-0.16,0.64);
          \node at (-0.6,0.8) {$1$};
          \draw (0,-0.8) circle (0.22);
          \draw (-0.16,-0.96)--(0.16,-0.64) (0.16,-0.96)--(-0.16,-0.64);
          \node at (-0.6,-0.8) {$2$};
          \node at (-0.32,0) {$-1$};
          \draw (1.86,-0.8) circle (0.22);
          \draw (1.70,-0.96)--(2.02,-0.64) (1.70,-0.64)--(2.02,-0.96);
          \draw (1.86,-0.58)--(0.22,0.8)  (0.22,-0.8)--(1.64,-0.8);
          \node at (0.93,-1.05) {$-\lambda$};
          \node at (1.50,0.4) {$1+\lambda$};
          \node at (2.38,-0.8) {$3$};
        \end{tikzpicture}
         & $\begin{array}{cc}
              \begin{pmatrix}
            0 & -1 & 1+\lambda\\
            -1 & 0 & -\lambda \\
            1+\lambda & -\lambda & 0
        \end{pmatrix}  \\
              ~ \\
              ~
         \end{array}$  \\
        \hline
    \end{tabular}
    \caption{Dynkin diagrams and Cartan matrices of type $D(2,1;\lambda)$}
\end{table}
Let $a_{ij}$ denote the $(i,j)$-entry of Cartan matrix $A(D_{\lambda})$ (resp. $A(D_{\lambda}^{\circ})$). The Lie superalgebra $\mathfrak{g}$ subject to the Cartan matrix $A=A(\mathfrak{g})$ has the following defining relations:
\begin{align}\label{LA:1}
    &[h_i,\,h_j]=0,\quad [e_i,\,f_j]=\delta_{ij}h_i, \\ \label{LA:2}
    &[h_i,\,e_j]=a_{ij}e_j, \quad [h_i,\,f_j]=-a_{ij}f_j, \\ \label{LA:3}
    &[e_i,\,e_j]=[f_i,\,f_j]=0,\quad \text{if}~~(\alpha_i,\alpha_j)=0, \\ \label{LA:4}
    &[e_i,\,[e_i,\,e_j]]=[f_i,\,[f_i,\,f_j]]=0,\quad \text{if}~~\mathfrak{g}=D_{\lambda},\ a_{ij}\neq 0,\ \text{and}~~i\neq j, \\ 
    &a_{13}[[e_1,\,e_2],\,e_3]=a_{12}[[e_1,\,e_3],\,e_2],\quad \text{if}~~\mathfrak{g}=D_{\lambda}^{\circ},\nonumber \\ \label{LA:5}
    &a_{13}[[f_1,\,f_2],\,f_3]=a_{12}[[f_1,\,f_3],\,f_2],\quad \text{if}~~\mathfrak{g}=D_{\lambda}^{\circ}.
\end{align}
By its definition, there is a Chevalley involution $\omega$ of $\mathfrak{g}$ such that
    \begin{gather*}
        \omega(e_i)=-f_i,\quad \omega(f_i)=-(-1)^{|i|}e_i,\quad \omega(h_i)=-h_i,\quad i\in I. 
    \end{gather*}

The functions
$\alpha_1,\alpha_2,\alpha_3\in\mathfrak{h}^{\ast}$ are simple roots of $\mathfrak{g}=D_{\lambda}$ (resp. $D_{\lambda}^{\circ}$) such that $\alpha_j(h_i)=a_{ij}$. Let $\mathcal{R}^+$ be the positive root system of $\mathfrak{g}$. Put $\mathcal{R}=\mathcal{R}^+\cup-\mathcal{R}^+$. We also set
$$\mathcal{Q}=\mathbb{Z}\alpha_1\oplus\mathbb{Z}\alpha_2\oplus\mathbb{Z}\alpha_3,\quad \mathcal{Q}^+=\mathbb{Z}_+\alpha_1\oplus\mathbb{Z}_+\alpha_2\oplus\mathbb{Z}_+\alpha_3\ \subset\ \mathfrak{h}^{\ast}. $$
The lattice $\mathcal{Q}$ is called the \textit{root lattice}. Then $\mathfrak{g}$ admits a root space decomposition subject to $\mathfrak{h}$:
\begin{gather*}
    \mathfrak{g}=\left(\oplus_{\alpha\in\mathcal{Q}^+\setminus\{0\}}\mathfrak{g}_{-\alpha}\right)\oplus \mathfrak{h}\oplus \left(\oplus_{\alpha\in\mathcal{Q}^+\setminus\{0\}}\mathfrak{g}_{\alpha}\right),
\end{gather*}
where $\mathfrak{g}_{\alpha}:=\{x\in\mathfrak{g}|[h_i,\,x]=\alpha_i(h)x~~\text{for all}~~i\in I\}$. For any $\alpha\in \mathcal{Q}_+\cup-\mathcal{Q}_+$ and $\alpha\neq 0$, we have
\begin{enumerate}
    \item If $\alpha\notin\mathcal{R}\cup\{0\}$, then $\mathfrak{g}_{\alpha}$ is trivial (i.e. $\dim \mathfrak{g}_{\alpha}=0$). 
    \item If $\alpha\in\mathcal{R}$, then $\dim\mathfrak{g}_{\alpha}=1$. 
\end{enumerate}

We will describe the \textit{Chevalley basis} of $\mathfrak{g}$ in these two cases. Before that, we introduce the root vectors $e_{\alpha},f_{\alpha}$ for $\alpha\in\mathcal{R}^+$ via the following rules. Set for any $\alpha=\sum_{i}m_i\alpha_i\in\mathcal{R}$
\begin{gather*}
    \operatorname{ht}\alpha=\sum_im_i,\quad \operatorname{ms}\alpha=\operatorname{min}\{i|m_i\neq 0\}. 
\end{gather*}
For $\alpha,\beta\in \mathcal{R}^+$, we call $\alpha\prec\beta$ if one of the conditions holds,
\begin{itemize}
    \item[(1)]  $\operatorname{ms}\alpha<\operatorname{ms}\beta$;
    \item[(2)]  $\operatorname{ms}\alpha=\operatorname{ms}\beta$ and $\beta-\alpha\in\mathcal{R}^+$;
    \item[(3)]  $\operatorname{ms}\alpha=\operatorname{ms}\beta$ and $\operatorname{ms}(\alpha-\operatorname{ms}\alpha)<\operatorname{ms}(\beta-\operatorname{ms}\beta)$. 
\end{itemize}
We also call $\alpha\prec\beta$ if $-\beta\prec-\alpha$ for negative roots $\alpha,\beta$. 
Then we have

 \textbf{Case 1}: $\mathfrak{g}=D_{\lambda}$. In this case, 
\begin{gather*}
    \mathcal{R}^+=\{\alpha_1\prec \alpha_1+\alpha_2\prec \alpha_1+\alpha_2+\alpha_3\prec \alpha_1+2\alpha_2+\alpha_3\prec \alpha_2\prec \alpha_2+\alpha_3\prec\alpha_3\}.
\end{gather*}
Denote
\begin{align*}
    &e_{\alpha_i}=e_i,\quad e_{\alpha_1+\alpha_2}=[e_1,\,e_2],\quad e_{\alpha_2+\alpha_3}=[e_2,\,e_3], \\
    &e_{\alpha_1+\alpha_2+\alpha_3}=[[e_1,e_2],\,e_3],\quad e_{\alpha_1+2\alpha_2+\alpha_3}=[[[e_1,e_2],\,e_3],\,e_2].
\end{align*}

 \textbf{Case 2}: $\mathfrak{g}=D_{\lambda}^{\circ}$. In this case, 
\begin{gather*}
    \mathcal{R}^+=\{\alpha_1\prec \alpha_1+\alpha_2\prec\alpha_1+\alpha_3\prec \alpha_1+\alpha_2+\alpha_3\prec \alpha_2\prec \alpha_2+\alpha_3\prec\alpha_3\}.
\end{gather*}
Denote
\begin{align*}
    &e_{\alpha_i}=e_i,\quad e_{\alpha_1+\alpha_2}=[e_1,\,e_2],\quad e_{\alpha_2+\alpha_3}=[e_2,\,e_3], \\
    &e_{\alpha_1+\alpha_3}=[e_1,\,e_3],\quad e_{\alpha_1+\alpha_2+\alpha_3}=[[e_1,e_2],\,e_3].
\end{align*}

Define $f_{\alpha}:=\omega(e_{\alpha})$ for $\alpha\in\mathcal{R}^+$. Therefore, the set $\{e_{\alpha},f_{\alpha},h_i|\alpha\in\mathcal{R}^+,i\in I\}$ forms an ordered basis of $\mathfrak{g}$.

\subsection{Half Casimir element}\label{se:simpleLiesuper:HC}
As known, the Cartan matrix $A$ is always symmetrizable, that is, there is a diagonal matrix $D=\operatorname{diag}\{d_1,d_2,d_3\}$ such that $DA$ is symmetric. Choose
\begin{align*}
    &D_{\lambda}:~~d_1=1,\ d_2=-1,\ d_3=\lambda, \\
    &D_{\lambda}^{\circ}:~~d_1=d_2=d_3=1.
\end{align*}
Recall that \cite[Section 4]{VdL89}, there is a non-degenerate $\mathbb{C}$-value bilinear form $(\,\cdot\,,\,\cdot\,)$ on $\mathfrak{g}$ such that
\begin{enumerate}
    \item $(\,\cdot\,,\,\cdot\,)$ is supersymmetric, i.e., 
    \begin{gather*}
        (x,y)=(-1)^{|x||y|}(y,x)\quad \text{for homogeneous }x,y\in\mathfrak{g};
    \end{gather*}

    \item $(\,\cdot\,,\,\cdot\,)$ is invariant, i.e., 
    \begin{gather*}
        (x,[y,\,z])=([x,\,y],z)\quad \text{for }~~x,y,z\in\mathfrak{g};
    \end{gather*}

    \item The restriction of $(\,\cdot\,,\,\cdot\,)$ to $\mathfrak{h}$ is non-degenerate. Moreover,  $(h_i,h_j)=a_{ij}/d_j$ for $i,j\in I$; 

    \item $(\mathfrak{g}_{\alpha},\mathfrak{g}_{\beta})=0$ if $\alpha\neq -\beta$;
    
    \item $[x,\,y]=(x,y)h_{\alpha}$ for $x\in\mathfrak{g}_{\alpha}$, $y\in \mathfrak{g}_{-\alpha}$, $\alpha\in\mathcal{R}^+$, where $h_{\alpha}=\sum_i m_i d_ih_i$ if $\alpha=\sum_i m_i\alpha_i$. 
\end{enumerate}
This form provides a non-degenerate bilinear form on the root lattice $\mathcal{Q}$ given by
\begin{gather*}
    (\alpha_i,\alpha_j)=d_ia_{ij}=d_id_j(h_i,h_j). 
\end{gather*}

Choose two bases $h^{(k)}$ and $h_{(k)}$of $\mathfrak{h}$ dual each other, i.e. $(h^{(k)},h_{(l)})=\delta_{ij}$ for $i,j\in I$. 
Then for each $i\in I$, we have
\begin{align*}
    &h_i=\sum_k ( h^{(k)},h_i) h_{(k)}=\sum_k ( h_{(k)},h_i) h^{(k)}, \\
    &[h^{(k)},\,e_i]=d_i( h^{(k)},h_i) e_i,\quad [h^{(k)},\,f_i]=-d_i( h^{(k)},h_i) f_i,\\
    &[h_{(k)},\,e_i]=d_i(h_{(k)},h_i) e_i,\quad [h_{(k)},\,f_i]=-d_i( h_{(k)},h_i) f_i.
\end{align*}

Choose a class of suitable pairs $(\xi_{\alpha}^+,\xi_{\alpha}^-)$ such that
$(\xi_{\alpha}^+e_{\alpha},\xi_{\alpha}^-f_{\alpha})=1$ for all $\alpha\in\mathcal{R}^+$. 
Denote $$\tilde{e}_{\alpha}=\xi_{\alpha}^+e_{\alpha},\quad \tilde{f}_{\alpha}=\xi_{\alpha}^-f_{\alpha}.$$ In particular,
$\tilde{e}_{\alpha_i}=e_i,\ \tilde{f}_{\alpha_i}=d_if_i$. 
For any $\alpha=\sum_{i\in I}a_i\alpha_i\in\mathcal{Q}$, we set $|\alpha|=\sum_{i\in I}a_i|i|$. Then we have
\begin{lemma}\label{roots-tensor}
    For $\alpha,\beta\in\mathcal{R}^+$ and $z\in\mathfrak{g}_{\beta-\alpha}$, the following equation holds in $\mathfrak{g}\otimes \mathfrak{g}$,
    \begin{gather*}
        \tilde{f}_{\alpha}\otimes [z,\,\tilde{e}_{\alpha}]=(-1)^{|\alpha|+|\beta|}[\tilde{f}_{\alpha},\,z]\otimes \tilde{e}_{\alpha}.
    \end{gather*}
\end{lemma}
\begin{proof}
    It can be checked as in the usual case; see \cite[Lemma 2.4]{Ka90}. 
    
\end{proof}

Introduce a Casimir element $\Omega$ of $\mathfrak{g}\otimes \mathfrak{g}$:
\begin{gather*}
    \Omega=\sum_{i\in I} h^{(i)}\otimes h_{(i)}+\sum_{\alpha\in\mathcal{R}^+}\tilde{f}_{\alpha}\otimes \tilde{e}_{\alpha}+\sum_{\alpha\in\mathcal{R}^+}(-1)^{|\alpha|}\tilde{e}_{\alpha}\otimes \tilde{f}_{\alpha}.
\end{gather*}
It is easy to see that $\Omega$ is even and satisfies the following properties:
\begin{enumerate}
    \item $[x\otimes 1+1\otimes x,\,\Omega]=0$ for all $x\in\mathfrak{g}$. Moreover, let $\mu$ be the $\mathbb{Z}_2$-graded multiplication map such that
    \begin{gather*}
        \mu(x\otimes y)=(-1)^{|x||y|}xy\quad\text{for all}~~x,y\in \mathfrak{g}. 
    \end{gather*}
    Then $\mu(\Omega)$ is contained in the center of the universal enveloping superalgebra $U(\mathfrak{g})$.  
    \item $\Omega=P(\Omega)$, where $P$ is the $\mathbb{Z}_2$-graded permutation operator such that 
    $$P(x\otimes y)=(-1)^{|x||y|}y\otimes x\quad \text{for homogeneous}~~x,y.$$
\end{enumerate}
The ``half" Casimir element of $\mathfrak{g}\otimes \mathfrak{g}$ is given by
\begin{gather*}
    \Omega^+=\sum_{i\in I} h^{(i)}\otimes h_{(i)}+\sum_{\alpha\in\mathcal{R}^+}\tilde{f}_{\alpha}\otimes \tilde{e}_{\alpha}.
\end{gather*}
Although $\Omega^+$ is not supercommutative with coproducts of the generators $e_i$ $f_i$ and $h_i$, it exhibits the following simple commutation relations. 
\begin{lemma}\label{Half-casimir}
    The following equations hold,
    \begin{align}\label{half-casimir-1}
        &[h_i\otimes 1+1\otimes h_i,\,\Omega^+]=0, \\ \label{half-casimir-2}
        &[e_i\otimes 1+1\otimes e_i,\,\Omega^+]=-e_i\otimes d_ih_i, \\ \label{half-casimir-3}
        &[f_i\otimes 1+1\otimes f_i,\,\Omega^+]=d_ih_i\otimes f_i.
    \end{align}
\end{lemma}
\begin{proof}
    For relation \eqref{half-casimir-1}, we have
    \begin{gather*}
        [h_i\otimes 1+1\otimes h_i,\,\Omega^+]=\sum_{\alpha\in\mathcal{R}^+}\left([h_i,\,\tilde{f}_{\alpha}]\otimes \tilde{e}_{\alpha}+\tilde{f}_{\alpha}\otimes [h_i,\,\tilde{e}_{\alpha}]\right)=0.
    \end{gather*}

    We now check relation \eqref{half-casimir-2} and a similar calculation can be used to \eqref{half-casimir-3}. In fact, 
    \begin{align*}
        [e_i\otimes 1,\,\Omega^+]&=\sum_{j\in I}[e_i,\,h^{(j)}]\otimes h_{(j)}+\sum_{\alpha\in\mathcal{R}^+}[e_i,\,\tilde{f}_{\alpha}]\otimes \tilde{e}_{\alpha} \\
        &=-e_i\otimes d_ih_i+d_ih_i\otimes e_i+\sum_{\alpha\in\mathcal{R}^+\setminus\{\alpha_i\}}[e_i,\,\tilde{f}_{\alpha}]\otimes \tilde{e}_{\alpha} \\
        &=-e_i\otimes d_ih_i+d_ih_i\otimes e_i-\sum_{\alpha\in\mathcal{R}^+\setminus\{\alpha_i\}}(-1)^{|\alpha||\alpha_i|}\tilde{f}_{\alpha}\otimes [e_i,\,\tilde{e}_{\alpha}] \\
        &=-e_i\otimes d_ih_i-[1\otimes e_i,\,\Omega^+], 
    \end{align*}
by Lemma \ref{roots-tensor}. 
    
\end{proof}
The element $\Omega^+$ will be used in Section \ref{se:hopfstructure}.

\subsection{Affine Lie superalgebras and current Lie superalgebras}\label{se:simpleLiesuper:AF}
Let $\widehat{\mathfrak{g}}$ be the affinization of $\mathfrak{g}$. Put $\hat{I}=I\cup \{0\}$. The affine Lie superalgebra $\widehat{\mathfrak{g}}$ possess the generators $\{e_i,f_i,h_i|i\in \hat{I}\}$ with the Serre presentation constructed in \cite[Section 5]{Ya99}. Here, $\widehat{\mathfrak{g}}$ is associated with the following affine Cartan matrix $\widehat{A}(\mathfrak{g})=(a_{ij})_{i,j\in\hat{I}}$, which can be obtained from $A(\mathfrak{g})$ by adding $0$-th row and $0$-th column. 
\begin{equation*}
    \widehat{A}(D_{\lambda})=\begin{pmatrix}
        2 & 0 & -1 & 0 \\
        0 & 2 & -1 & 0 \\
        -1-\lambda & 1 & 0 & \lambda \\
        0 & 0 & -1 & 2
    \end{pmatrix},
    \qquad \widehat{A}(D_{\lambda}^{\circ})=\begin{pmatrix}
        0 & -\lambda & 1+\lambda & -1 \\
        -\lambda & 0 & -1 & 1+\lambda \\
        1+\lambda & -1 & 0 & -\lambda \\
        -1 & 1+\lambda & -\lambda & 0
    \end{pmatrix}.
\end{equation*}

Within the framework of \cite{VdL89}, $\widehat{\mathfrak{g}}$ also can also be defined as the central $\mathfrak{c}$ extension of the loop superalgebra $\mathcal{L}\mathfrak{g}:=\mathfrak{g}\otimes \mathbb{C}[t,t^{-1}]$ (without derivation) with the supercommutator
\begin{gather}\label{supercommutator}
   [x^{(r)}+k\mathfrak{c},\,y^{(s)}+l\mathfrak{c}]=[x,\,y]^{(r+s)}+r\delta_{r,-s}(x,y)\mathfrak{c}, 
\end{gather}
where $x,y\in \mathfrak{g}$, $k,l\in\mathbb{Z}$, and we write $z^{(r)}=z\otimes t^r$ for $z\in\mathfrak{g}$ and $r\in\mathbb{Z}$. The following proposition provides the equivalence between the two definitions of affinization of $\mathfrak{g}$ proved in \cite[Theorem 6.9]{VdL89}.
\begin{proposition}\label{Iso-AFg}
    The mapping $\imath$:
    \begin{align*}
    &e_i\mapsto e_i^{(0)},\quad f_i\mapsto f_i^{(0)},\quad h_i\mapsto h_i^{(0)},\quad i\in I, \\
    &e_0\mapsto f_{\theta}^{(1)},\quad f_0\mapsto c_0e_{\theta}^{(-1)},\quad h_0\mapsto d_0^{-1}\mathfrak{c}-\sum_{i\in I}c_ih_i^{(0)},
    \end{align*}
is an isomorphism between the superalgebras $\widehat{\mathfrak{g}}$ and $\mathcal{L}\mathfrak{g}\oplus \mathbb{Z}\mathfrak{c}$, where  
$\theta$ is the highest root of $\mathfrak{g}$ and 
\begin{align*}
    &c_0=d_0^2,\ c_1=\frac{d_1}{d_0},\ c_2=\frac{2d_2}{d_0},\ c_3=\frac{d_3}{d_0},~~~~\text{for}~~\mathfrak{g}=D_{\lambda}~~\text{and}~~d_0=-1-\lambda, \\
    &c_i=d_i=1,~~~~\text{for}~~\mathfrak{g}=D_{\lambda}^{\circ}~~\text{and}~~i\in \hat{I}. 
\end{align*}
\end{proposition}

The universal enveloping algebra of the polynomial current Lie superalgebra $\mathfrak{g}[u]$ can be viewed as a limit form of the universal enveloping algebra of the loop superalgebra $\mathcal{L}\mathfrak{g}$; one may refer to the usual case in reference \cite[Section 2.1]{LWZ25}. Define recursively for $z\in\mathfrak{g}$, $r\in\mathbb{Z}$ and $m\in\mathbb{Z}_+$,
\begin{gather*}
    z^{(r,0)}:=z^{(r)},\quad z^{(r,m+1)}:=z^{(r+1,m)}-z^{(r,m)}.
\end{gather*}
It follows that
\begin{gather}\label{loop-recurse-eq1}
    z^{(r,m)}=\sum_{k=0}^m(-1)^{m-k}\binom{m}{k}z^{(r)}=z\otimes t^r(t-1)^m\in\mathcal{L}\mathfrak{g}. 
\end{gather}
Clearly, the set of all elements $z^{(r,m)}$ spans $\mathcal{L}\mathfrak{g}$. Introduce a filtration
\begin{gather*}
    \mathsf{w}_0\supset \mathsf{w}_1\supset\mathsf{w}_2\supset \cdots\supset \mathsf{w}_m\supset \cdots
\end{gather*}
on the enveloping algebra $U(\mathcal{L}\mathfrak{g})$ by setting $\deg z^{(r,m)}=m$ and $\deg (xy)=\deg x+\deg y$ for homogeneous $x,y\in U(\mathcal{L}\mathfrak{g})$, where each $\mathsf{w}_m$ ($m\in\mathbb{Z}_+$) denotes the span of the monomials
\begin{gather*}
    z_1^{(r_1,m_1)}z_2^{(r_2,m_2)}\cdots z_t^{(r_p,m_p)}
\end{gather*}
for $z_i\in\mathfrak{g}$, $r_i\in\mathbb{Z}$, $m_i\in\mathbb{Z}_+$, $1\leqslant i\leqslant p$, $p\geqslant 0$ such that $m_1+m_2+\cdots +m_p\geqslant m$. The graded algebra associated with this filtration is given by
\begin{gather*}
    \operatorname{gr}_{\mathsf{w}}U(\mathcal{L}\mathfrak{g}):=\bigoplus_{m=0}^{\infty} \mathsf{w}_m/\mathsf{w}_{m+1}.
\end{gather*}
By relation \eqref{loop-recurse-eq1}, we have for $z\in\mathfrak{g}$, $m\in\mathbb{Z}_+$, 
\begin{gather}\label{loop-recurse-eq2}
    z^{(r,m)}\equiv z^{(0,m)}\quad \text{mod}~~\mathsf{w}_{m+1},\quad \forall r\in\mathbb{Z}.
\end{gather}

\begin{proposition}
    There exists an isomorphism $\varsigma$ of superalgebras between $U(\mathfrak{g}[u])$ and $\operatorname{gr}_{\mathsf{w}}U(\mathcal{L}\mathfrak{g})$ such that
    \begin{gather*}
        z\otimes u^m\mapsto \bar{z}^{(0,m)},\quad z\in\mathfrak{g},\ m\in\mathbb{Z}_+.
    \end{gather*}
    Here, $\bar{z}^{(0,m)}$ denotes the image of $z^{(0,m)}$ on the component $\mathsf{s}_m/\mathsf{s}_{m+1}$.
\end{proposition}
\begin{proof}
    The mapping $\varsigma$ is an epimorphism as a direct consequence of \eqref{supercommutator} and \eqref{loop-recurse-eq2}. Furthermore, the injectivity of $\varsigma$ can be deduced from the PBW theorem of $U(\mathfrak{g}[u])$ and $U(\mathcal{L}\mathfrak{g})$. 
    
\end{proof}

\medskip

\section{Quantum affine(loop) superalgebras of $D(2,1;\lambda)$}\label{se:quantumAFsuper}
As one of the main results of Heckenberger-Spill-Torrielli-Yamane's work \cite{HSTY08}, quantum affine superalgebras of $D(2,1;\lambda)$ have two equivalent presentations---Drinfeld-Jimbo presentation and Drinfeld (second) presentation. In this section, we will review the isomorphism theorem between these two presentations and construct a PBW basis in terms of Drinfeld current generators. 

Given $\hbar\in\mathbb{C}\setminus\mathbb{Z}\pi\sqrt{-1}$. denote for $n\in\mathbb{C}$,
\begin{gather*}
    q^n:=\exp(n\hbar),\quad [n]_q:=\frac{q^n-q^{-n}}{q-q^{-1}}. 
\end{gather*}
In particular, set $q=q^1$. From here to the end of this paper, assume that 
\begin{gather*}
    q^{kn}\neq 1,\quad  n\in\{1,\lambda,\lambda+1\},\ k\in \mathbb{N}.
\end{gather*}
Introduce the super $q$-bracket $[\,\cdot\,,\,\cdot\,]_a$ for $a\in\mathbb{C}(q)\setminus\{0\}$ by
\begin{gather*}
    [x,\,y]_a=xy-(-1)^{|x||y|}ayx
\end{gather*}
for homogeneous $x,y$. 
Clearly, $[x,\,y]_1=[x,\,y]$.

\subsection{Drinfeld-Jimbo presentation $U_q(\widehat{\mathfrak{g}})$}\label{se:quantumAFsuper:DJ}
We begin with the definition of the quantum affine superalgebra associated with Chevalley generators, i.e. Drinfeld-Jimbo presentation. The complete Serre-like relations follow from \cite{Ya99}. 
\begin{definition}
    The \textit{quantum affine superalgebra} $U_q(\widehat{\mathfrak{g}})$ is an associative superalgebra over $\mathbb{C}(q)$ generated by the elements
    \begin{gather*}
        E_i,\ F_i,\ K_i^{\pm 1}:=\exp(H_i\hbar)\quad\text{for}~~i\in \hat{I},
    \end{gather*}
    with $\mathbb{Z}_2$-grading $|E_i|=|F_i|=|i|$ and $|K_i^{\pm 1}|=\bar{0}$. The defining relations are given as follows.
    \begin{align}\label{DJ:1}
        &K_i^{\pm 1}K_i^{\mp 1}=1,\quad K_iK_j=K_jK_i, \\ \label{DJ:2}
        &K_iE_jK_i^{-1}=q^{(\alpha_i,\alpha_j)}E_j,\quad K_iF_jK_i^{-1}=q^{-(\alpha_i,\alpha_j)}F_j, \\ \label{DJ:3}
        &[E_i,\,F_j]=\delta_{ij}\frac{K_i-K_i^{-1}}{q-q^{-1}}, \\ \label{DJ:4}
        &[E_i,\,E_j]=[F_i,\,F_j]=0,\quad \text{if}~~(\alpha_i,\alpha_j)=0, \\ \label{DJ:5}
        &[\![E_i,\,[\![E_i,\,E_j]\!] ]\!]=[\![F_i,\,[\![F_i,\,F_j]\!] ]\!]=0,\quad \text{if}~~\mathfrak{g}=D_{\lambda},\ (\alpha_i,\alpha_j)\neq 0,\ \text{and}~~ i\neq j, \\ \label{DJ:6}
        &[\![  [\![ [\![E_2,\,E_0]\!],\,[\![E_2,\,E_1]\!]  ]\!]  ,\,[\![E_2,\,E_3]\!] ]\!] =[\lambda]_q[\![  [\![ [\![E_2,\,E_0]\!],\,[\![E_2,\,E_3]\!]  ]\!]  ,\,[\![E_2,\,E_1]\!] ]\!],\quad \text{if}~~\mathfrak{g}=D_{\lambda}, \\  \label{DJ:7}
        &[\![  [\![ [\![F_2,\,F_0]\!],\,[\![F_2,\,F_1]\!]  ]\!]  ,\,[\![F_2,\,F_3]\!] ]\!] =[\lambda]_q[\![  [\![ [\![F_2,\,F_0]\!],\,[\![F_2,\,F_3]\!]  ]\!]  ,\,[\![F_2,\,F_1]\!] ]\!],\quad \text{if}~~\mathfrak{g}=D_{\lambda}, \\ \label{DJ:8}
        &[(\alpha_i,\alpha_k)]_q[\![ [\![E_i,\,E_j]\!],\,E_k]\!]=[(\alpha_i,\alpha_j)]_q[\![ [\![E_i,\,E_k]\!],\,E_j]\!],\quad \text{if}~~\mathfrak{g}=D_{\lambda}^{\circ},\ i<j<k, \\ \label{DJ:9}
        &[(\alpha_i,\alpha_k)]_q[\![ [\![F_i,\,F_j]\!],\,F_k]\!]=[(\alpha_i,\alpha_j)]_q[\![ [\![F_i,\,F_k]\!],\,F_j]\!],\quad \text{if}~~\mathfrak{g}=D_{\lambda}^{\circ},\ i<j<k, 
    \end{align}
    where $[\![x,\,y]\!]:=[x,\,y]_{q^{-(\alpha,\beta)}}$ if $K_ixK_i^{-1}=q^{(\alpha_i,\alpha)}x$ and $K_iyK_i^{-1}=q^{(\alpha_i,\beta)}y$. 
\end{definition}

Let $\widehat{\mathfrak{h}}$ be the Cartan subalgebra and $\delta$ the null root of $\widehat{\mathfrak{g}}$. Put $\alpha_0=\delta-\theta$ and $\widehat{\mathcal{Q}}=\sum_{i\in \hat{I}}\mathbb{Z}\alpha_i$. We can extend the notation $K_{\alpha}$ for all $\alpha\in\widehat{\mathcal{Q}}$ by
\begin{gather*}
    K_{\alpha}=\prod_{i\in\hat{I}}K_i^{m_i},\quad \text{if}~~\alpha=\sum_{i\in\hat{I}}m_i\alpha_i\in\widehat{\mathcal{Q}}.
\end{gather*}
Note that $K_{\delta}^{\pm 1}$ are central elements of $U_q(\widehat{\mathfrak{g}})$. 

The superalgebra $U_q(\widehat{\mathfrak{g}})$ is a Hopf superalgebra endowed with coproduct $\triangle_q:\ U_q(\widehat{\mathfrak{g}})\rightarrow U_q(\widehat{\mathfrak{g}})\otimes U_q(\widehat{\mathfrak{g}})$, counit $\epsilon_q:\ U_q(\widehat{\mathfrak{g}})\rightarrow \mathbb{C}$ and antipode $\mathcal{S}_q:\ U_q(\widehat{\mathfrak{g}})\rightarrow U_q(\widehat{\mathfrak{g}})$ given by
\begin{align*}
    &\triangle_q(K_i^{\pm 1})=K_i^{\pm 1}\otimes K_i^{\pm 1},\quad \triangle_q(K_{\delta}^{\pm 1})=K_{\delta}^{\pm 1}\otimes K_{\delta}^{\pm 1}, \\
    &\triangle_q(E_i)=E_i\otimes 1+K_i\otimes E_i,\quad \triangle_q(F_i)=F_i\otimes K_i^{-1}+1\otimes F_i, \\
    &\epsilon_q(K_i^{\pm 1})=\epsilon_q(K_{\delta}^{\pm 1})=1,\quad \epsilon_q(E_i)=\epsilon_q(F_i)=0, \\
    &\mathcal{S}_q(E_i)=-K_i^{-1}E_i,\quad \mathcal{S}_q(F_i)=-F_iK_i,\quad \mathcal{S}_q(K_i)=K_i^{-1},\quad \mathcal{S}_q(K_{\delta})=K_{\delta}^{-1}.
\end{align*}

\begin{remark}
    Here, we do not need to use Yamane's extension of $U_q(\widehat{\mathfrak{g}})$ as presented in \cite[Section 6.1]{Ya99} since we employ the $\mathbb{Z}_2$-graded tensor product multiplication rule.
\end{remark}

$U_q(\widehat{\mathfrak{g}})$ can be regarded as a $\mathbb{C}[[\hbar]]$-associative superalgebra with generators $H_i$, $E_i$, $F_i$ for $i\in \hat{I}$, which allows us to specialize $\hbar$. We denote it by $U_{\hbar}(\widehat{\mathfrak{g}})$. From the equivalent statement of \cite[Lemma 6.6.1]{Ya99}, it follows that
\begin{proposition}\label{special:DJ}
    There is an isomorphism of superalgebras between the quotient $U_{\hbar}(\widehat{\mathfrak{g}})/\hbar U_{\hbar}(\widehat{\mathfrak{g}})$ and $U(\widehat{\mathfrak{g}})$ such that
    \begin{gather*}
        \bar{E}_i\mapsto e_i,\quad \bar{F}_i\mapsto d_if_i,\quad \bar{H_i}\mapsto d_ih_i,\quad i\in\hat{I},
    \end{gather*}
    where $\bar{x}$ denotes the image of $x$ on $U_{\hbar}(\widehat{\mathfrak{g}})/\hbar U_{\hbar}(\widehat{\mathfrak{g}})$. 
    Moreover, $U_{\hbar}(\widehat{\mathfrak{g}})$ is isomorphic to $U(\widehat{\mathfrak{g}})[[\hbar]]$ as $\mathbb{C}[[\hbar]]$-module. 
\end{proposition}
\begin{remark}
    The process under the isomorphism in Proposition \ref{special:DJ} is termed taking the \textit{classical limit}. 
\end{remark}

\subsection{Drinfeld presentation $\mathcal{U}_q(\widehat{\mathfrak{g}})$}\label{se:quantumAFsuper:DS}
In \cite{HSTY08}, the authors derived the Drinfeld presentation of quantum affine superlagebra of $D(2,1;\lambda)$ using Beck's approach \cite{Be94}. We restate this definition below. 
\begin{definition}
    The \textit{quantum affine superalgebra} $\mathcal{U}_q(\widehat{\mathfrak{g}})$ is an associative superalgebra over $\mathbb{C}(q)$ generated by the elements
    \begin{gather*}
        \mathcal{X}_{i,r}^{\pm},\ \mathcal{H}_{i,\ell},\ \mathcal{K}_i^{\pm 1},\ \mathcal{C}^{\pm\frac{1}{2}},\quad i\in I,\ r\in\mathbb{Z},\ \ell\in\mathbb{Z}\setminus\{0\},
    \end{gather*}
    with $\mathbb{Z}_2$-grading $|\mathcal{X}_{i,r}^{\pm}|=|i|$ and $|\mathcal{H}_{i,\ell}|=|\mathcal{K}_i^{\pm 1}|=|\mathcal{C}^{\pm\frac{1}{2}}|=\bar{0}$. The defining relations are given as follows.
    \begin{align}\label{DS:1}
        &\mathcal{C}^{\pm\frac{1}{2}}~~\text{are central elements},\quad \mathcal{C}^{\pm\frac{1}{2}}\mathcal{C}^{\mp\frac{1}{2}}=\mathcal{K}_i^{\pm 1}K_i^{\mp 1}=1, \\ \label{DS:2}
        &\mathcal{K}_i\mathcal{K}_j=\mathcal{K}_j\mathcal{K}_i,\quad [K_i,\,\mathcal{H}_{j,\ell}]=0, \\ \label{DS:3}
        &[\mathcal{H}_{i,k},\,\mathcal{H}_{j,\ell}]=\delta_{k,-\ell}\frac{[k(\alpha_i,\alpha_j)]_q}{k}\frac{\mathcal{C}^k-\mathcal{C}^{-k}}{q-q^{-1}}, \\ \label{DS:4}
        &\mathcal{K}_i\mathcal{X}_{j,r}^{\pm}\mathcal{K}_i^{-1}
        =q^{\pm(\alpha_i,\alpha_j)}\mathcal{X}_{j,r}^{\pm}, \\ \label{DS:5}
        &[\mathcal{H}_{i,\ell},\,\mathcal{X}_{j,r}^{\pm}]=\pm \frac{[k(\alpha_i,\alpha_j)]_q}{k} \mathcal{C}^{\mp\frac{\operatorname{abs}(\ell)}{2}}\mathcal{X}_{j,\ell+r}^{\pm}, \\ \label{DS:6}
        &[\mathcal{X}_{i,r}^+,\,\mathcal{X}_{j,s}^-]=\delta_{ij}\frac{\mathcal{C}^{\frac{r-s}{2}}\Phi_{i,r+s}^+-\mathcal{C}^{\frac{s-r}{2}}\Phi_{i,r+s}^-}{q-q^{-1}}, \\ \label{DS:7}
        &[\![\mathcal{X}_{i,r\mp 1}^{\pm},\,\mathcal{X}_{j,s}^{\pm}]\!]+(-1)^{|i||j|}[\![\mathcal{X}_{j,s\mp 1}^{\pm},\,\mathcal{X}_{i,r}^{\pm}]\!]=0,\quad \text{if}~~(\alpha_i,\alpha_j)\neq 0, \\ \label{DS:8}
        &[\mathcal{X}_{i,r}^{\pm},\,\mathcal{X}_{j,s}^{\pm}]=0,\quad \text{if}~~(\alpha_i,\alpha_j)=0,\\ \label{DS:9}
        &[\![\mathcal{X}_{i,k}^{\pm},\,[\![\mathcal{X}_{i,r}^{\pm},\,\mathcal{X}_{j,s}^{\pm} ]\!] ]\!]+[\![\mathcal{X}_{i,r}^{\pm},\,[\![\mathcal{X}_{i,k}^{\pm},\,\mathcal{X}_{j,s}^{\pm} ]\!] ]\!]=0,\quad \text{if}~~\mathfrak{g}=D_{\lambda},\ (\alpha_i,\alpha_j)\neq 0,\ \text{and}~~i\neq j, \\ \label{DS:10}
        &[(\alpha_1,\alpha_3)]_q[\![ [\![\mathcal{X}_{1,k}^{\pm},\,\mathcal{X}_{2,r}^{\pm}]\!],\,\mathcal{X}_{3,s}^{\pm}]\!]=[(\alpha_1,\alpha_2)]_q[\![ [\![\mathcal{X}_{1,k}^{\pm},\,\mathcal{X}_{3,s}^{\pm}]\!],\,\mathcal{X}_{2,r}^{\pm}]\!],\quad \text{if}~~\mathfrak{g}=D_{\lambda}^{\circ},
    \end{align}
    where the function $\operatorname{abs}(\ell)$ denotes the absolute value of $\ell$, and $\Phi_{i,\pm k}^{\pm}$ is determined by the formal power series
    \begin{gather}\label{DS:11}
        \sum_{k\in\mathbb{Z}_+}\Phi_{i,\pm k}^{\pm}z^{-k}=\mathcal{K}_i^{\pm 1}\exp\left(\pm(q-q^{-1})\sum_{\ell\in\mathbb{N}}\mathcal{H}_{i,\pm\ell}z^{-\ell}\right).
    \end{gather}
\end{definition}

The next proposition establishes the existence of the isomorphism between the two presentations of quantum affine superalgebra associated to $D(2,1;\lambda)$. We adopt the inverse map of the isomorphism given in \cite[Theorem 6.6]{HSTY08}, with slight refinements. 
\begin{proposition}\label{Isom-QASA}
    There is an isomorphism $\Pi$ of superalgebras from the Drinfeld-Jimbo presentation $U_q(\widehat{\mathfrak{g}})$ to the Drinfeld presentation $\mathcal{U}_q(\widehat{\mathfrak{g}})$ such that
    \begin{equation*}
    \begin{split}
        &E_i\mapsto \mathcal{X}_{i,0}^+,\quad F_i\mapsto\mathcal{X}_{i,0}^-,\quad K_i^{\pm 1}\mapsto \mathcal{K}_i^{\pm 1},\quad  i\in I, \\
        &\mathfrak{g}=D_{\lambda}\begin{cases}
            E_0\mapsto \Pi(K_0)[\![ [\![ [\![\mathcal{X}_{1,1}^-,\,\mathcal{X}_{2,0}^-]\!] ,\,\mathcal{X}_{3,0}^-]\!] ,\,\mathcal{X}_{2,0}^-]\!],\\ 
            F_0\mapsto q^{2(\lambda+1)}[\lambda]_q[\lambda+1]_q\Pi(K_0)^{-1}[\![ [\![ [\![\mathcal{X}_{1,1}^-,\,\mathcal{X}_{2,0}^-]\!] ,\,\mathcal{X}_{3,0}^-]\!] ,\,\mathcal{X}_{2,0}^-]\!],\\
            K_0\mapsto \mathcal{C}\mathcal{K}_1^{-1}\mathcal{K}_2^{-2}\mathcal{K}_3^{-1},
        \end{cases} \\
        &\mathfrak{g}=D_{\lambda}^{\circ}\begin{cases}
            E_0\mapsto \Pi(K_0)[\![ [\![\mathcal{X}_{1,1}^-,\,\mathcal{X}_{2,0}^-]\!] ,\,\mathcal{X}_{3,0}^-]\!] ,\\ 
            F_0\mapsto -\Pi(K_0)^{-1}[\![ [\![\mathcal{X}_{1,1}^-,\,\mathcal{X}_{2,0}^-]\!] ,\,\mathcal{X}_{3,0}^-]\!],\\
            K_0\mapsto \mathcal{C}\mathcal{K}_1^{-1}\mathcal{K}_2^{-1}\mathcal{K}_3^{-1}.
        \end{cases}
    \end{split}
    \end{equation*}
\end{proposition}

\begin{remark}
Denote $\mathcal{K}_{\theta}=\Pi(K_{\theta})$ and 
\begin{equation*}
\mathfrak{X}_{\theta}^{\mp}=\begin{cases}
    [\![ [\![ [\![\mathcal{X}_{1,\pm 1}^{\mp},\,\mathcal{X}_{2,0}^{\mp}]\!] ,\,\mathcal{X}_{3,0}^{\mp}]\!] ,\,\mathcal{X}_{2,0}^{\mp}]\!],  & \text{if }~~\mathfrak{g}=D_{\lambda}, \\
    [\![ [\![\mathcal{X}_{1,\pm 1}^{\mp},\,\mathcal{X}_{2,0}^{\mp}]\!] , \,\mathcal{X}_{3,0}^{\mp}]\!]&\text{if }~~\mathfrak{g}=D_{\lambda}^{\circ}. 
\end{cases}
\end{equation*}
A straightforward calculation yields $$[\mathfrak{X}_{\theta}^-,\,\mathfrak{X}_{\theta}^+]=p_{\mathfrak{g}}\frac{\mathcal{C}\mathcal{K}_{\theta}^{-1}-\mathcal{C}^{-1}\mathcal{K}_{\theta}}{q-q^{1}},$$ 
where the constant $p_{\mathfrak{g}}\in\mathbb{C}(q)$ is given by
\begin{equation*}
    p_{\mathfrak{g}}=\begin{cases}
        q^{2(\lambda+1)}[\lambda]_q[\lambda+1]_q, &\text{if }~~\mathfrak{g}=D_{\lambda}, \\ 
        -1, & \text{if }~~\mathfrak{g}=D_{\lambda}^{\circ}. 
    \end{cases}
\end{equation*}
By  definition, the morphism fixes $E_i,F_i$ for $i\in I$, $K_j$ for $j\in\hat{I}$, while scaling
\begin{gather*}
    E_0\mapsto g^+E_0,\quad F_0\mapsto g^-F_0\qquad(g^+,g^-\in\mathbb{C}(q),\ g^+g^-=p_{\mathfrak{g}}),
\end{gather*}
automatically preserves the Drinfeld-Jimbo relations and thus defines an 
automorphism of $U(\widehat{\mathfrak{g}})$. 
\end{remark}

\vspace{1em}

Proposition \ref{Isom-QASA} allows us to define $\mathcal{K}_i=\exp(\mathcal{H}_{i,0}\hbar)$. In this way, we regard  $\mathcal{U}_q(\widehat{\mathfrak{g}})$ as a $\mathbb{C}[[\hbar]]$-associative superalgebra with generators $\mathcal{X}_{i,r}^{\pm}$ and $\mathcal{H}_{i,r}$ for $i\in I$, $r\in\mathbb{Z}$, denoted by $\mathcal{U}_{\hbar}(\widehat{\mathfrak{g}})$. The quantum loop superalgebra $\mathcal{U}_{\hbar}(\mathcal{L}\mathfrak{g})$ is then defined as
the quotient of $\mathcal{U}_{\hbar}(\widehat{\mathfrak{g}})$ by the relations $\mathcal{C}^{\pm\frac{1}{2}}=1$. The generators of $\mathcal{U}_{\hbar}(\mathcal{L}\mathfrak{g})$ satisfy relations \eqref{DS:7}--\eqref{DS:10} and the subsequent lemma. 
\begin{lemma}
    The following equalities hold in $\mathcal{U}_{\hbar}(\mathcal{L}\mathfrak{g})$, 
    \begin{align}\label{Dh:1}
        &[\mathcal{H}_{i,r},\,\mathcal{H}_{j,s}]=0, \\ \label{Dh:2}
        &[\mathcal{H}_{i,0},\,\mathcal{X}_{j,s}^{\pm}]=\pm (\alpha_i,\alpha_j)\mathcal{X}_{j,s}^{\pm}, \\ \label{Dh:3}
        &[\mathcal{H}_{i,r},\,\mathcal{X}_{j,s}^{\pm}]=\pm \frac{[r(\alpha_i,\alpha_j)]_q}{r} \mathcal{X}_{j,r+s}^{\pm}, \\ \label{Dh:4}
        &[\mathcal{X}_{i,r}^+,\,\mathcal{X}_{j,s}^-]=\delta_{ij}\frac{\Phi_{i,r+s}^+-\Phi_{i,r+s}^-}{q-q^{-1}},
    \end{align}
    where we still use the notations $\mathcal{X}_{i,r}^{\pm}$, $\mathcal{H}_{i,r}$, $\Phi_{i,\pm k}^{\pm}$ {\rm(}$i\in I$, $r\in\mathbb{Z}$, $k\in\mathbb{Z}_+${\rm)} to denote the image of the generators of $\mathcal{U}_{\hbar}(\hat{\mathfrak{g}})$. 
\end{lemma}

As a consequence of Proposition \ref{Iso-AFg}, \ref{special:DJ} and \ref{Isom-QASA},  the following diagram commutes: 

\vspace{0.4em}
\centerline{
   $\xymatrix@=1cm{ 
  U_{\hbar}'(\widehat{\mathfrak{g}}) \ar@{->>}[r] \ar[d]_{\simeq} 
    & U_{\hbar}'(\widehat{\mathfrak{g}})/\hbar U_{\hbar}'(\widehat{\mathfrak{g}}) \ar[r]_{~~~~\simeq} \ar[d]^{\simeq} & U'(\widehat{\mathfrak{g}}) \ar[d]^{\simeq} \\
  \mathcal{U}_{\hbar}(\mathcal{L}\mathfrak{g}) \ar@{->>}[r] 
    & \mathcal{U}_{\hbar}(\mathcal{L}\mathfrak{g})/\hbar \mathcal{U}_{\hbar}(\mathcal{L}\mathfrak{g}) \ar[r]^{~~~~\simeq} & U(\mathcal{L}\mathfrak{g}),
}$
}

\vspace{0.4em}
\noindent
where $U_{\hbar}'(\widehat{\mathfrak{g}})$ (resp. $U'(\widehat{\mathfrak{g}})$) denotes the quotient of $U_{\hbar}(\widehat{\mathfrak{g}})$ (resp. $U(\widehat{\mathfrak{g}})$) by the ideal generated by $K_{\theta}-1$ (resp. $\mathfrak{c}$). 
In other words, we have 
\begin{proposition}\label{special:Drinfeld}
    There is a $\mathbb{C}$-superalgebraic isomorphism $\mathcal{U}_{\hbar}(\mathcal{L}\mathfrak{g})/\hbar\mathcal{U}_{\hbar}(\mathcal{L}\mathfrak{g})\rightarrow U(\mathcal{L}\mathfrak{g})$ such that
    \begin{gather*}
        \bar{\mathcal{X}}_{i,r}^+\mapsto e_i^{(r)},\quad  \bar{\mathcal{X}}_{i,r}^-\mapsto d_if_i^{(r)},\quad \bar{\mathcal{H}}_{i,r}\mapsto d_ih_i^{(r)},
    \end{gather*}
    where $\bar{\mathcal{X}}$ denotes the image of $\mathcal{X}$ on the quotient $\mathcal{U}_{\hbar}(\mathcal{L}\mathfrak{g})/\hbar\mathcal{U}_{\hbar}(\mathcal{L}\mathfrak{g})$. 
    Moreover, $\mathcal{U}_{\hbar}(\mathcal{L}\mathfrak{g})$ is isomorphic to $U(\mathcal{L}\mathfrak{g})[[\hbar]]$ as $\mathbb{C}[[\hbar]]$-module. 
\end{proposition}

For an $\alpha\in\mathcal{R}^+$ with a composition $\alpha=\alpha_{i_1}+\cdots+\alpha_{i_l}$ and a sequence $\overline{\bf r}=(r_{i_1},\cdots,r_{i_l})\in \mathbb{Z}^{l}$ such that $r=r_1+\ldots+r_l$, define
\begin{gather*}
    \mathcal{X}_{\alpha,\overline{\bf r}}^{\pm}:=[\![\ldots[\![\mathcal{X}_{i_1,r_1}^{\pm},\,\mathcal{X}_{i_2,r_2}^{\pm}]\!],\ldots,\mathcal{X}_{i_l,r_l}^{\pm}]\!],
\end{gather*}
if the root vector $e_{\alpha}=[\ldots[e_{i_1},\,e_{i_2}],\ldots,e_{i_l}]\in\mathfrak{g}$ as described in Section 2.1. Set $\overline{\bf r}^{0}=(r,0,\ldots,0)$.

\begin{corollary}\label{PBW:Uh}
    Fix some ordering on the countable set
    $\{\mathcal{X}_{\alpha,\overline{\bf r}}^{\pm},\ \mathcal{H}_{i,r}\}_{\alpha\in\mathcal{R^+},i\in I, r\in\mathbb{Z}}$. 
    The set of all ordered monomials in all $\mathcal{X}_{\alpha,\overline{\bf r}^{0}}^{\pm}$ and $\mathcal{H}_{i,r}$ {\rm (}with the powers of odd elements not greater than 1{\rm )} forms a basis for $\mathcal{U}_{\hbar}(\mathcal{L}\mathfrak{g})$ over $\mathbb{C}[[\hbar]]$.
\end{corollary}

Corollary \ref{PBW:Uh} provides a PBW basis for the quantum loop superalgebra $\mathcal{U}_{\hbar}(\mathcal{L}\mathfrak{g})$, which will be used in Section \ref{se:superYangian:fromto}.

\medskip
\section{Super Yangians of $D(2,1;\lambda)$}\label{se:superYangian}
In this section, we will introduce the super Yangians of $D(2,1;\lambda)$ in terms of Drinfeld presentation, which has not appeared in prior works. We use Guay-Ma-Conner's \cite{CG15,GM12} degeneration to construct a PBW basis for our definition, originating from \cite{LWZ25} for twisted Yangians of split type. Finally, we give some isomorphisms related to these new super Yangians. 
\subsection{Definition of the super Yangians}\label{se:superYangian:def}
Let us now define the Drinfeld presentations of the super Yangians of $D(2,1;\lambda)$ subject to a parameter $\hbar\neq 0$. 
\begin{definition}\label{Def-Y}
    Drinfeld presentation of the super Yangian $Y_{\hbar}(\mathfrak{g})$ is an associative superalgebra over $\mathbb{C}[\hbar]$ generated by the elements
    \begin{gather*}
        x_{i,m}^{\pm},\ h_{i,m},\quad i\in I,\ m\in\mathbb{Z}_+,
    \end{gather*}
    with $\mathbb{Z}_2$-grading $|x_{i,m}^{\pm}|=|i|$ and $|h_{i,m}|=\bar{0}$. The defining relations are given as follows. 
    \begin{align}\label{SY:1}
        &[h_{i,m},\,h_{i,n}]=0, \\ \label{SY:2}
        &[h_{i,0},\,x_{j,n}^{\pm}]=\pm(\alpha_i,\alpha_j)x_{j,n}^{\pm}, \\ \label{SY:3}
        &[h_{i,m+1},\,x_{j,n}^{\pm}]=[h_{i,m},\,x_{j,n+1}^{\pm}]\pm\frac{\hbar(\alpha_i,\alpha_j)}{2}\{h_{i,m},\,x_{j,n}^{\pm}\}, \\ \label{SY:4}
        &[x_{i,m}^+,\,x_{j,n}^-]=\delta_{ij}h_{i,m+n}, \\ \label{SY:5}
        &[x_{i,m+1}^{\pm},\,x_{j,n}^{\pm}]=[x_{i,m}^{\pm},\,x_{j,n+1}^{\pm}]\pm\frac{\hbar(\alpha_i,\alpha_j)}{2}\{x_{i,m}^{\pm},\,x_{j,n}^{\pm}\}, \\ \label{SY:6}
        &[x_{i,m}^{\pm},\,x_{j,n}^{\pm}]=0,\quad \text{if}~~(\alpha_i,\alpha_j)=0, \\ \label{SY:7}
        &[x_{i,l}^{\pm},\,[x_{i,m}^{\pm},\,x_{j,n}^{\pm}]]+[x_{i,m}^{\pm},\,[x_{i,l}^{\pm},\,x_{j,n}^{\pm}]]=0,\quad \text{if}~~\mathfrak{g}=D_x,\ (\alpha_i,\alpha_j)\neq 0,\ \text{and}~~i\neq j, \\ \label{SY:8}
        &(\alpha_1,\alpha_3)[[x_{1,l}^{\pm},\,x_{2,m}^{\pm}],\,x_{3,n}^{\pm}]=(\alpha_1,\alpha_2)[[x_{1,l}^{\pm},\,x_{3,n}^{\pm}],\,x_{2,m}^{\pm}],\quad \text{if}~~\mathfrak{g}=D_x^{\circ}. 
    \end{align}
\end{definition}

\begin{remark}\label{hbar}
    In the definition above, the parameter $\hbar$ can assume any non-zero complex number. Nevertheless, for the sake of convenience in this paper, we require that $\hbar$ remain consistent with that employed in Section \ref{se:quantumAFsuper}. The superalgebras $Y_{\hbar}(\mathfrak{g})$ for all possible $\hbar$ are pairwise isomorphic. More specifically, an isomorphism $Y_{\hbar'}(\mathfrak{g})\rightarrow Y_{\hbar}(\mathfrak{g})$ can be defined by
    \begin{gather*}
        x_{i,m}^{\pm}\mapsto (\hbar'/\hbar)^m x_{i,m}^{\pm},\quad h_{i,m}\mapsto (\hbar'/\hbar)^m h_{i,m}.
    \end{gather*} 
\end{remark}

Introduce a filtration on the superalgebra $Y_{\hbar}(\mathfrak{g})$ by setting
$$\deg x_{i,m}^{\pm}=\deg h_{i,m}=m,\quad \text{and}\quad \deg (xy)=\deg x+\deg y~~\text{for homogeneous}~~x,y\in Y_{\hbar}(\mathfrak{g}).$$
Denote
\begin{align*}
    &Y_{\hbar}^{[p]}(\mathfrak{g}):=\operatorname{span}_{\mathbb{C}[\hbar]}\{\,x\in Y_{\hbar}(\mathfrak{g})\,|\,\deg x\leqslant p\,\},\quad p\geqslant0, \\
    &\operatorname{gr}^{[p]}Y_{\hbar}(\mathfrak{g}):=Y_{\hbar}^{[p]}(\mathfrak{g})/Y_{\hbar}^{[p-1]}(\mathfrak{g}),\quad p>0, \quad \operatorname{gr}^{[0]}Y_{\hbar}(\mathfrak{g})=Y_{\hbar}^{[0]}(\mathfrak{g}).
\end{align*}
The graded algebra associated with this filtration is defined as
\begin{gather*}
    \operatorname{gr}Y_{\hbar}(\mathfrak{g})=\oplus_{p\geqslant 0}\operatorname{gr}^{[p]}Y_{\hbar}(\mathfrak{g}).
\end{gather*}

For an $\alpha\in\mathcal{R}^+$ with a composition $\alpha=\alpha_{i_1}+\cdots+\alpha_{i_l}$ and a sequence $\overline{\bf m}=(m_{i_1},\cdots,m_{i_l})\in \mathbb{Z}_+^{l}$ such that $m=m_1+\ldots+m_l$, define
\begin{gather*}
    x_{\alpha,\overline{\bf m}}^{\pm}:=[\ldots[x_{i_1,m_1}^{\pm},\,x_{i_2,m_2}^{\pm}],\ldots,x_{i_l,m_l}^{\pm}],
\end{gather*}
if the root vector $e_{\alpha}=[\ldots[e_{i_1},\,e_{i_2}],\ldots,e_{i_l}]\in\mathfrak{g}$ as described in Section 2.1. Observe that $x_{\alpha,\overline{\bf m}}^{\pm}\in Y_{\hbar}^{[m]}(\mathfrak{g})$. We need to show that for any $\overline{\bf m}\in\mathbb{Z}_+^l$,
\begin{gather*}
    x_{\alpha,\overline{\bf m}}^{\pm}-x_{\alpha,\overline{\bf m}^0}^{\pm}\in Y_{\hbar}^{[m-1]}(\mathfrak{g}),
\end{gather*}
where $\overline{\bf m}^0=(m,0,\ldots,0)$. We prove it case by case. Denote $s(\mathcal{P})$ to be the degree of a polynomial $\mathcal{P}\in\mathfrak{g}$ in elements $e_i$. 
\begin{enumerate}
    \item[(1)] $s(e_{\alpha})=1$. This is obvious.

    \item[(2)] $s(e_{\alpha})=2$. By relation \eqref{SY:5}, for $m'+n'=m+n$ and $\alpha=\alpha_{i_1}+\alpha_{i_2}$,
    \begin{gather*}
        x_{\alpha,(m',n')}^{\pm}=[x_{i_1,m'}^{\pm},\,x_{j_1,n'}^{\pm}]\equiv x_{\alpha,(m,n)}^{\pm}\quad \text{mod}~~~Y_{\hbar}^{[m+n-1]}(\mathfrak{g}).
    \end{gather*}

    \item[(3)] $s(e_{\alpha})=3$ and $\mathfrak{g}=D_{\lambda}$. Simple calculation gives
    \begin{align*}
        x_{\alpha_1+\alpha_2+\alpha_3,(m_1,m_2,m_3)}^{\pm}&=[[x_{1,m_1}^{\pm},\,x_{2,m_2}^{\pm}],\,x_{3,m_3}^{\pm}]\equiv [[x_{1,m_1+m_2}^{\pm},\,x_{2,0}^{\pm}],\,x_{3,m_3}^{\pm}] \\
        &\equiv[x_{1,m_1+m_2}^{\pm},\,[x_{2,m_3}^{\pm},\,x_{3,0}^{\pm}]]+[[x_{1,m_1+m_2}^{\pm},\,x_{3,m_3}^{\pm}],\,x_{2,0}^{\pm}] \\
        &\equiv [[x_{1,m_1+m_2}^{\pm},\,x_{2,m_3}^{\pm}],\,x_{3,m_0}^{\pm}]+[x_{2,m_3}^{\pm},\,[x_{1,m_1+m_2}^{\pm},\,x_{3,0}^{\pm}]] \\
        &\equiv [[x_{1,m_1+m_2+m_3}^{\pm},\,x_{2,0}^{\pm}],\,x_{3,m_0}^{\pm}] \quad \hbox{mod} \quad Y_{\hbar}^{[m_1+m_2+m_3-1]}(\mathfrak{g}),
    \end{align*}
    by relations \eqref{SY:5} and \eqref{SY:6}. 

    \item[(4)] $s(e_{\alpha})=3$ and $\mathfrak{g}=D_{\lambda}^{\circ}$. 
    Following relations \eqref{SY:8}, one has for any $l,m,n\in\mathbb{Z}_+$, 
    \begin{gather*}
        (1+\lambda)[x_{1,l}^{\pm},\,[x_{2,m}^{\pm},\,x_{3,n}^{\pm}]]=\lambda[[x_{1,l}^{\pm},\,x_{3,n}^{\pm}],\,x_{2,m}^{\pm}].
    \end{gather*}
    Then we deduce
    \begin{align*}
        x_{\alpha_1+\alpha_2+\alpha_3,(m_1,m_2,m_3)}^{\pm}
        &\equiv [x_{1,m_1+m_2}^{\pm},\,[x_{2,0}^{\pm},\,x_{3,m_3}^{\pm}]]-[[x_{1,m_1+m_2}^{\pm},\,x_{3,m_3}^{\pm}],\,x_{2,0}^{\pm}] \\
        &\equiv -\frac{1}{1+\lambda}[[x_{1,m_1+m_2+m_3}^{\pm},\,x_{3,0}^{\pm}],\,x_{2,0}^{\pm}] \\
        &=[[x_{1,m_1+m_2+m_3}^{\pm},\,x_{2,0}^{\pm}],\,x_{3,0}^{\pm}]\quad \hbox{mod} \quad Y_{\hbar}^{[m_1+m_2+m_3-1]}(\mathfrak{g}). 
    \end{align*}

    \item[(5)] $s(e_{\alpha})=4$. According to relation \eqref{SY:7} to get for $m,n\in\mathbb{Z}_+$,
    $$[x_{2,m}^{\pm},[x_{2,m}^{\pm},\,x_{1,n}^{\pm}]]=[[x_{1,n}^{\pm},\,x_{2,m}^{\pm}],\,x_{2,m}^{\pm}].$$
    Then we have for $\overline{\bf m}=(m_1,m_2,m_3,m_4)$,
    \begin{align*}
        x_{\alpha_1+2\alpha_2+\alpha_3,\overline{\bf m}}^{\pm}&=[[[x_{1,m_1}^{\pm},\,x_{2,m_2}^{\pm}],\,x_{3,m_3}^{\pm}],\,x_{2,m_4}^{\pm}] \\
        &\equiv [[[x_{1,m_1+m_2}^{\pm},\,x_{2,0}^{\pm}],\,x_{2,m_4}^{\pm}],\,x_{3,m_3}^{\pm}]+[[x_{1,m_1+m_2}^{\pm},\,x_{2,0}^{\pm}],\,[x_{3,m_3}^{\pm},\,x_{2,m_4}^{\pm}]] \\
        &\equiv[[[x_{1,m-m_3}^{\pm},\,x_{2,0}^{\pm}],\,x_{2,0}^{\pm}],\,x_{3,m_3}^{\pm}]+[[[x_{1,m_1+m_2}^{\pm},\,x_{2,0}^{\pm}],\,x_{3,m_3+m_4}^{\pm}],\,x_{2,0}^{\pm}] \\
        &\equiv[[[x_{1,m}^{\pm},\,x_{2,0}^{\pm}],\,x_{3,0}^{\pm}],\,x_{2,m_4}^{\pm}] \quad \hbox{mod} \quad Y_{\hbar}^{[m-1]}(\mathfrak{g}).
    \end{align*}
   \end{enumerate}

Moreover, by modulo $Y_{\hbar}^{[m+n-1]}(\mathfrak{g})$, the commutation relations in $Y_{\hbar}^{[m+n]}(\mathfrak{g})$ coincide with those in $U(\mathfrak{g}[u])[\hbar]$:
\begin{align*}
    &[h_{i,m},\,h_{i,n}]\equiv 0,\quad [h_{i,m},\,x_{j,n}^{\pm}]\equiv \pm(\alpha_i,\alpha_j)x_{j,m+n}^{\pm}, \\
    &[x_{i,m}^+,\,x_{j,n}^-]\equiv \delta_{ij}h_{i,m+n},\\
    &[x_{i,m+n}^{\pm},\,x_{j,0}^{\pm}]\equiv\cdots\equiv [x_{i,m}^{\pm},\,x_{j,n}^{\pm}]\equiv\cdots\equiv[x_{i,0}^{\pm},\,x_{j,m+n}^{\pm}]. 
\end{align*}
This implies that the mapping $U(\mathfrak{g}[u])[\hbar]\rightarrow \operatorname{gr}Y_{\hbar}(\mathfrak{g})$ is an epimorphism over $\mathbb{C}[\hbar]$. 

As a consequence of PBW theorem for $U(\mathfrak{g}[u])[\hbar]$, we formulate the following proposition.
\begin{proposition}\label{Y:PBW}
    Fix some ordering on the countable set 
    $\{x_{\alpha,\overline{\bf m}}^{\pm},h_{i,m}\}_{\alpha\in\mathcal{R}^+,i\in I,m\in\mathbb{Z}_+}$.
    The set of all ordered monomials in all $x_{\alpha,\overline{\bf m}^{0}}^{\pm}$ and $h_{i,m}$ {\rm(}with the powers of odd elements not greater than 1{\rm)} spans $Y_{\hbar}(\mathfrak{g})$ over $\mathbb{C}[\hbar]$. 
\end{proposition}

\subsection{From quantum loop superalgebras to super Yangians}\label{se:superYangian:fromto}
Let $\pi$ be the composite of the following superalgebraic homomorphism sequence:
\begin{gather*}
    \mathcal{U}_{\hbar}(\mathcal{L}\mathfrak{g})\twoheadrightarrow \mathcal{U}_{\hbar}(\mathcal{L}\mathfrak{g})/\hbar\mathcal{U}_{\hbar}(\mathcal{L}\mathfrak{g})\xlongrightarrow{\simeq}U(\mathcal{L}\mathfrak{g}).
\end{gather*}
Let $W_m$($m\in\mathbb{Z}_+$) be the Lie ideal of $\mathcal{L}\mathfrak{g}$ spanned by the elements $z\otimes t^{r}(t-1)^m$ for all $z\in\mathfrak{g}$ and $r\in\mathbb{Z}$. Denote
\begin{gather*}
    U=\operatorname{span}\big\{\,\mathcal{H}_{i,r},\ \mathcal{X}_{\alpha,r}^{\pm}\,\big|\,\alpha\in\mathcal{R}^+,\ i\in I,\ r\in\mathbb{Z}\,\big\}.
\end{gather*}
Set $\mathcal{W}_m$ to be the $\mathbb{C}[[\hbar]]$-two-sided ideal of $\mathcal{U}_{\hbar}(\mathcal{L}\mathfrak{g})$ generated by $\pi^{-1}(W_m)\cap U$. Define $\mathbf{W}_m$ to be the sum of the ideals $\hbar^{m_0}\mathcal{W}_{m_1}\cdots \mathcal{W}_{m_l}$ such that $m_0+m_1+\cdots+m_l\geqslant m$ for all $m_i\in\mathbb{Z}_+$ and $0\leqslant l\leqslant m$. It is evident that $\mathbf{W}_0=\mathcal{U}_{\hbar}(\mathcal{L}\mathfrak{g})$. We will formulate the explicit form of the elements in $\mathbf{W}_m$. 
Denote recursively for $\alpha\in\mathcal{R}^+$, $i\in I$, $r\in\mathbb{Z}$, $m\in\mathbb{Z}_+$,
\begin{alignat*}{2}
&\mathsf{H}_{i,r;0}=\frac{\Phi_{i,r}^+-\Phi_{i,r}^-}{q-q^{-1}},  & \qquad\mathsf{H}_{i,r;m+1}=\mathsf{H}_{i,r+1;m}-\mathsf{H}_{i,r;m}, \\
       &\mathsf{X}_{\alpha,r;0}^{\pm}=\mathcal{X}_{\alpha,r}^{\pm},  & \qquad\mathsf{X}_{\alpha,r;m+1}^{\pm}=\mathsf{X}_{\alpha,r+1;m}^{\pm}-\mathsf{X}_{\alpha,r;m}^{\pm}.
\end{alignat*}
That is to say,
\begin{gather*}
    \mathsf{H}_{i,r;m}=\sum_{k=0}^m(-1)^{m-k}\binom{m}{k}\frac{\Phi_{i,r+k}^+-\Phi_{i,r+k}^-}{q-q^{-1}},\quad \mathsf{X}_{\alpha,r;m}^{\pm}=\sum_{k=0}^m(-1)^{m-k}\binom{m}{k}\mathcal{X}_{\alpha,r+k}^{\pm}.
\end{gather*}
As a similar statement in \cite{GM12,LWZ25}, any monomial $\mathcal{M}$ in $\mathbf{W}_m$ has the form
\begin{gather*}
    \mathcal{M}=f(\hbar)\hbar^{m_0}\mathcal{X}_{\alpha_1^-,r_1^-;m_1^-}^-\ldots \mathcal{X}_{\alpha_a^-,r_a^-;m_a^-}^-\mathcal{H}_{i_1,r_1^0,m_1^0}\cdots\mathcal{H}_{i_b,r_b^0,m_b^0}\mathcal{X}_{\alpha_1^+,r_1^+;m_1^+}^+\ldots \mathcal{X}_{\alpha_c^+,r_c^+;m_c^+}^+,
\end{gather*}
where $f(0)\neq 0$ and $m_0+m_1^-+\cdots+m_a^-+m_1^0+\cdots+m_b^0+m_1^++\cdots+m_c^+\geqslant m$. Put
\begin{gather*}
    \operatorname{gr}_{\mathbf{W}}\mathcal{U}_{\hbar}(\mathcal{L}\mathfrak{g})=\bigoplus_{m\in\mathbb{Z}_+}\mathbf{W}_{m}/\mathbf{W}_{m+1}. 
\end{gather*}

The following theorem is one of the main results of our paper. 
\begin{theorem}\label{main1}
    There exists a $\mathbb{C}[\hbar]$-isomorphism $\varphi$ from the super Yangian $Y_{\hbar}(\mathfrak{g})$ to the graded algebra $\operatorname{gr}_{\mathbf{W}}\mathcal{U}_{\hbar}(\mathcal{L}\mathfrak{g})$ such that for $i\in I$ and $m\in\mathbb{Z}_+$,
    \begin{gather*}
        h_{i,m}\mapsto \bar{\mathsf{H}}_{i,0;m},\quad x_{i,m}^{\pm}\mapsto \bar{\mathsf{X}}_{i,0;m}^{\pm},
    \end{gather*}
    where $\bar{x}$ denotes the image of $x\in\mathcal{U}_{\hbar}(\mathcal{L}\mathfrak{g})$ in $\operatorname{gr}_{\mathbf{W}}\mathcal{U}_{\hbar}(\mathcal{L}\mathfrak{g})$.
\end{theorem}

To prove Theorem \ref{main1}, we require some essential lemmas. For simplicity, we set $b_{ij}=(\alpha_i,\alpha_j)$ for all $i,j\in I$. 
\begin{lemma}\label{LM-number}
In $\mathbf{W}_1$, 
    $q-q^{-1}\equiv 2\hbar~~\text{mod}~~\mathbf{W}_2$. Moreover, in $\mathbf{W}_0$, $[n]_q\equiv n~~\text{mod}~~\mathbf{W}_1$.
\end{lemma}
\begin{proof}
    These congruence relations can be checked directly. 
    
\end{proof}

\begin{lemma}\label{LM0}
    As elements in $\mathbf{W}_m$,
    \begin{gather*}
    \mathsf{H}_{i,r;m}\equiv \mathsf{H}_{i,0;m},\quad
        \mathsf{X}_{\alpha,r;m}^{\pm}\equiv \mathsf{X}_{\alpha,0;m}^{\pm} \quad \hbox{mod} \quad \mathbf{W}_{m+1}.
    \end{gather*}
    
\end{lemma}
\begin{proof}
    It is enough to show that, for each $r$, $\mathsf{H}_{\alpha,r;m}- \mathsf{H}_{\alpha,0;m},\ 
        \mathsf{X}_{\alpha,r;m}^{\pm}-\mathsf{X}_{\alpha,0;m}^{\pm}$ are elements of $\mathbf{W}_{m+1}$. This can be verified easily by induction on $m$.  
        
\end{proof}

\begin{lemma}\label{LM1}
    The following relation holds for $i,j\in I$, $m,n\in\mathbb{Z}_+$, $r,s\in\mathbb{Z}$ in $\mathcal{U}_{\hbar}(\mathcal{L}\mathfrak{g})$,
    \begin{gather}\label{LM1:eq:1}
        [\mathsf{H}_{i,r;m},\,\mathsf{H}_{j,s;n}]=0.
    \end{gather}
\end{lemma}
\begin{proof}
    By differentiating both sides of relation \eqref{DS:11} with respect to the parameter $z$, we obtain for $k\in\mathbb{N}$, 
\begin{gather}\label{LM1:eq:2}
    \Phi_{i,\pm k}^{\pm}=\pm(q-q^{-1})\mathcal{K}_i^{\pm 1}\mathcal{H}_{i,\pm k}\pm(q-q^{-1})\sum_{p=1}^{k-1} \frac{p}{k}\mathcal{H}_{i,\pm p}\Phi_{i,\pm(k-p)}^{\pm}. 
\end{gather}
Then relation \eqref{LM1:eq:1} follows from \eqref{Dh:1} and \eqref{LM1:eq:2}.

\end{proof}

\begin{lemma}\label{LM2}
    The following relations hold for $i,j\in I$, $m,n\in\mathbb{Z}_+$, $r\in\mathbb{Z}$, $k\in\mathbb{N}$ in $\mathcal{U}_{\hbar}(\mathcal{L}\mathfrak{g})$,
    \begin{align}\label{LM2:eq:1}
        [\mathcal{H}_{i,0},\,\mathsf{X}_{j,r;n}^{\pm}]&=\pm b_{ij}\mathsf{X}_{j,r;n}^{\pm}, \\ 
         [\mathsf{H}_{i,k;m+1},\,\mathsf{X}_{j,s;n}^{\pm}]&=[\mathsf{H}_{i,k;m},\,\mathsf{X}_{j,s;n+1}^{\pm}]+\left(q^{\pm b_{ij}}-1\right)\{\mathsf{H}_{i,k;m},\,\mathsf{X}_{j,s+1;n}^{\pm }\}\nonumber \\ \label{LM2:eq:2}
            &\qquad+
            \left(q^{\pm b_{ij}}-1\right)\Big(\mathsf{X}_{j,s;n}^{\pm}\mathsf{H}_{i,k;m+1}-\mathsf{X}_{j,s;n+1}^{\pm}\mathsf{H}_{i,k;m}\Big).
    \end{align}
    Moreover, in $\mathbf{W}_{m+n+1}$, 
    \begin{gather}\label{LM2:eq:3}
        [\mathsf{H}_{i,0;m+1},\,\mathsf{X}_{j,r;n}^{\pm}]\equiv [\mathsf{H}_{i,0;m},\,\mathsf{X}_{j,r;n+1}^{\pm}]\pm\frac{b_{ij}\hbar}{2}\{\mathsf{H}_{i,0;m},\,\mathsf{X}_{j,r+1;n}^{\pm}\} \quad \hbox{mod} \quad \mathbf{W}_{m+n+2}.
    \end{gather}
    \end{lemma}
\begin{proof}
    Relation \eqref{LM2:eq:1} can be checked easily by \eqref{Dh:2}. For relation \eqref{LM2:eq:2},  
    applying $[\,\cdot\,,\,\mathcal{X}_{i,k-r+1}^-]$ or $[\mathcal{X}_{i,k-r}^+,\,\cdot\,]$ for $k>0$ to relation \eqref{DS:7}, we obtain for $k>0$,
    \begin{gather*}
    [\mathcal{K}_i^{-1}\Phi_{i,k+1}^+,\,\mathcal{X}_{j,s}^{\pm}]
    =q^{\pm b_{ij}}\mathcal{K}_i^{-1}\Phi_{i,k}^+\mathcal{X}_{j,s+1}^{\pm}-q^{\mp b_{ij}}\mathcal{X}_{j,s+1}^{\pm}\mathcal{K}_i^{-1}\Phi_{i,k}^+
    \end{gather*}
    by relation \eqref{Dh:4}. This equation can be rewritten by
    \begin{gather*}
        [\mathcal{K}_i^{-1}\Phi_{i,k+1}^+,\,\mathcal{X}_{j,s}^{\pm}]=\frac{q^{b_{ij}}+q^{-b_{ij}}}{2}[\mathcal{K}_i^{-1}\Phi_{i,k}^+,\,\mathcal{X}_{j,s+1}^{\pm}]\pm \frac{q^{b_{ij}}-q^{-b_{ij}}}{2}\left\{\mathcal{K}_i^{-1}\Phi_{i,k}^+,\,\mathcal{X}_{j,s+1}^{\pm}\right\}
    \end{gather*}
    Multiplying the above relation by $(q-q^{-1})^{-1}K_i$ on the left to get
    \begin{equation*}
        \begin{split}
            &[\mathsf{H}_{i,k+1;0},\,\mathsf{X}_{j,s;0}^{\pm}]=\frac{q^{b_{ij}}+q^{-b_{ij}}}{2}[\mathsf{H}_{i,k;0},\,\mathsf{X}_{j,s+1;0}^{\pm}]\pm\frac{q^{b_{ij}}-q^{-b_{ij}}}{2}\{\mathsf{H}_{i,k;0},\,\mathsf{X}_{j,s+1;0}^{\pm }\} \\
            &+\left(q^{\pm b_{ij}}-1\right)\mathsf{X}_{j,s;0}^{\pm}\mathsf{H}_{i,k+1;0}+\left(q^{\mp b_{ij}}-1\right)\mathsf{X}_{j,s+1;0}^{\pm}\mathsf{H}_{i,k;0}.
        \end{split}
    \end{equation*}
    Subtracting $[\mathsf{H}_{i,k;0},\,\mathsf{X}_{j,s;0}^{\pm}]$ on both sides, rearranging the terms and simplifying yields
    \begin{equation*}
        \begin{split}
            [\mathsf{H}_{i,k;1},\,\mathsf{X}_{j,s;0}^{\pm}]&=[\mathsf{H}_{i,k;0},\,\mathsf{X}_{j,s;1}^{\pm}]+\left(q^{\pm b_{ij}}-1\right)\{\mathsf{H}_{i,k;0},\,\mathsf{X}_{j,s+1;0}^{\pm }\} \\
            &\qquad+
            \left(q^{\pm b_{ij}}-1\right)\Big(\mathsf{X}_{j,s;0}^{\pm}\mathsf{H}_{i,k;1}-\mathsf{X}_{j,s;1}^{\pm}\mathsf{H}_{i,k;0}\Big).
        \end{split}
    \end{equation*}
    Then we have \eqref{LM2:eq:2} by induction on $m$ and $n$, respectively. 

    Observe that the terms in relation \eqref{LM2:eq:2} are contained in $\mathbf{W}_{m+n+1}$. So relation \eqref{LM2:eq:3} holds due to Lemma \ref{LM0} and 
    \begin{gather}\label{LM2:eq:4}
        q^{\pm b_{ij}}-1=\pm q^{\pm\frac{b_{ij}}{2}}\left[\frac{b_{ij}}{2}\right]_q(q-q^{-1})\equiv \pm \frac{b_{ij}}{2}\hbar~~~\text{mod}~~~\mathbf{W}_2.
    \end{gather}

\end{proof}

\begin{lemma}\label{LM3}
    The following relation holds for $i,j\in I$, $r,s\in\mathbb{Z}$, $m,n\in\mathbb{Z}_+$ in $\mathcal{U}_{\hbar}(\mathcal{L}\mathfrak{g})$,
    \begin{gather}\label{LM3:eq:1}
        [\mathsf{X}_{i,r;m+1}^{\pm},\,\mathsf{X}_{j,s;n}^{\pm}]=[\mathsf{X}_{i,r;m}^{\pm},\,\mathsf{X}_{j,s;n+1}^{\pm}]+\left(q^{\pm b_{ij}}-1\right)\left(\mathsf{X}_{i,r;m}^{\pm}\mathsf{X}_{j,s+1;n}^{\pm}+\mathsf{X}_{j,s;n}^{\pm}\mathsf{X}_{i,r+1;m}^{\pm}\right).
    \end{gather}
    Moreover, in $\mathbf{W}_{m+n+1}$, 
    \begin{gather}\label{LM3:eq:2}
        [\mathsf{X}_{i,r;m+1}^{\pm},\,\mathsf{X}_{j,s;n}^{\pm}]\equiv [\mathsf{X}_{i,r;m}^{\pm},\,\mathsf{X}_{j,s;n+1}^{\pm}]\pm\frac{b_{ij}\hbar}{2}\left(\mathsf{X}_{i,r;m}^{\pm}\mathsf{X}_{j,s+1;n}^{\pm}+\mathsf{X}_{j,s;n}^{\pm}\mathsf{X}_{i,r+1;m}^{\pm}\right) ~~ \hbox{mod}~~ \mathbf{W}_{m+n+2}.
    \end{gather}
   \end{lemma}
\begin{proof}
    Relation \eqref{DS:7} can be rewritten as
    \begin{equation}\label{LM3:eq:3}
            [\mathcal{X}_{i,r+1}^{\pm},\,\mathcal{X}_{j,s}^{\pm}]+(-1)^{|i||j|}\left(1-q^{\pm b_{ij}}\right)\mathcal{X}_{j,s}^{\pm}\mathcal{X}_{i,r+1}^{\pm} =[\mathcal{X}_{i,r}^{\pm},\,\mathcal{X}_{j,s+1}^{\pm}]+\left(q^{\pm b_{ij}}-1\right)\mathcal{X}_{i,r}^{\pm}\mathcal{X}_{j,s+1}^{\pm}.
    \end{equation}
    Subtracting $[\mathcal{X}_{i,r}^{\pm},\,\mathcal{X}_{j,s}^{\pm}]$ from both sides of \eqref{LM3:eq:3} to get
     \begin{equation*}
            [\mathsf{X}_{i,r;1}^{\pm},\,\mathsf{X}_{j,s;0}^{\pm}]+(-1)^{|i||j|}\left(1-q^{\pm b_{ij}}\right)\mathsf{X}_{j,s;0}^{\pm}\mathsf{X}_{i,r+1;0}^{\pm} =[\mathsf{X}_{i,r;0}^{\pm},\,\mathsf{X}_{j,s;1}^{\pm}]+\left(q^{\pm b_{ij}}-1\right)\mathsf{X}_{i,r;0}^{\pm}\mathsf{X}_{j,s+1;0}^{\pm}.
    \end{equation*}
    Using induction on $m,n\geqslant 0$ and rearranging the terms, we deduce \eqref{LM3:eq:1}. Furthermore, \eqref{LM3:eq:2} follows from \eqref{LM3:eq:1} in terms of \eqref{LM2:eq:4}.

\end{proof}

\begin{lemma}\label{LM4}
    The following relations hold in $\mathcal{U}_{\hbar}(\mathcal{L}\mathfrak{g})$,
    \begin{align}\label{LM4:eq:1}
        &[\mathsf{X}_{i,r;m}^{\pm},\,\mathsf{X}_{j,s;n}^{\pm}]=0,\quad b_{ij}=0, \\ \label{LM4:eq:2}
        &[\mathsf{X}_{i,k;l}^{\pm},\,[\mathsf{X}_{i,r;m}^{\pm},\,\mathsf{X}_{j,s;n}^{\pm}]]+[\mathsf{X}_{i,r;m}^{\pm},\,[\mathsf{X}_{i,k;l}^{\pm},\,\mathsf{X}_{j,s;n}^{\pm}]]=0,\quad b_{ii}=0, 
     \end{align}
     Moreover, in $\mathbf{W}_{l+m+n}$, we have for $b_{ii}\neq 0$,
     \begin{gather}\label{LM4:eq:3}
         [\mathsf{X}_{i,k;l}^{\pm},\,[\mathsf{X}_{i,r;m}^{\pm},\,\mathsf{X}_{j,s;n}^{\pm}]]+[\mathsf{X}_{i,r;m}^{\pm},\,[\mathsf{X}_{i,k;l}^{\pm},\,\mathsf{X}_{j,s;n}^{\pm}]]\equiv 0
     \end{gather}
     modulo $\mathbf{W}_{l+m+n+1}$.
\end{lemma}

\begin{proof}
    We only prove relation \eqref{LM4:eq:3}, and relations \eqref{LM4:eq:1} and \eqref{LM4:eq:2} can be checked more easily. By transposing and simplifying relation \eqref{DS:9} for $b_{ii}\neq 0$, we have
    \begin{align*}
        &[\mathcal{X}_{i,k}^{\pm},\,[\mathcal{X}_{i,r}^{\pm},\,\mathcal{X}_{j,s}^{\pm}]]+[\mathcal{X}_{i,r},\,[\mathcal{X}_{i,k}^{\pm},\,\mathcal{X}_{j,s}^{\pm}]] \nonumber
 \\ \label{LM4:eq:3}
 &=\left(q^{\frac{b_{ij}}{2}}-q^{-\frac{b_{ij}}{2}}\right)^2\Big([\mathcal{X}_{i,r}^{\pm},\,\mathcal{X}_{j,s}^{\pm}]\mathcal{X}_{i,k}^{\pm}+[\mathcal{X}_{i,k}^{\pm},\,\mathcal{X}_{j,s}^{\pm}]\mathcal{X}_{i,r}^{\pm}+\mathcal{X}_{j,s}^{\pm}\{\mathcal{X}_{i,k}^{\pm},\,\mathcal{X}_{i,r}^{\pm}\}\Big).
    \end{align*}
    By using induction respctive on $l,m,n\geqslant 0$ and taking modulo $\mathbf{W}_{l+m+n+1}$, we have \eqref{LM4:eq:3}. 
    
\end{proof}

\begin{lemma}\label{LM5}
    The following relation holds for $k,r,s\in\mathbb{Z}$ in $\mathcal{U}_{\hbar}(\mathcal{L}\mathfrak{g})$,
    \begin{equation}\label{LM5:eq:1}
        \begin{split}
            &\quad~[1+\lambda]_q[[\mathcal{X}_{1,k}^{\pm},\,\mathcal{X}_{2,r}^{\pm}],\,\mathcal{X}_{3,s}^{\pm}]+[[\mathcal{X}_{1,k}^{\pm},\,\mathcal{X}_{3,s}^{\pm}],\,\mathcal{X}_{2,r}^{\pm}] \\
 &=([\lambda]_q+1-[1+\lambda]_q)\mathcal{X}_{3,s}^{\pm}[\mathcal{X}_{1,k}^{\pm},\,\mathcal{X}_{2,r}^{\pm}]+([1+\lambda]_q-[\lambda]_q-1)\mathcal{X}_{2,r}^{\pm}[\mathcal{X}_{1,k}^{\pm},\,\mathcal{X}_{3,s}^{\pm}] \\
 &\quad+([1+\lambda]_q+1-[2+\lambda]_q)[\mathcal{X}_{2,r}^{\pm},\,\mathcal{X}_{3,s}^{\pm}]\mathcal{X}_{1,k}^{\pm} \\
 &\quad+([2+\lambda]_q+[\lambda]_q-2[1+\lambda]_q)\mathcal{X}_{2,r}^{\pm}\mathcal{X}_{3,s}^{\pm}\mathcal{X}_{1,k}^{\pm}-(2+[\lambda]_q-[2+\lambda]_q)\mathcal{X}_{3,s}^{\pm}\mathcal{X}_{2,r}^{\pm}\mathcal{X}_{1,k}^{\pm}.
        \end{split}
    \end{equation}
    Moreover, in $\mathbf{W}_{l+m+n}$, we have for $l,m,n\geqslant 0$, 
    \begin{gather}\label{LM5:eq:2}
        (1+\lambda)[[\mathsf{X}_{1,k;l}^{\pm},\,\mathsf{X}_{2,r;m}^{\pm}],\,\mathsf{X}_{3,s;n}^{\pm}]+[[\mathsf{X}_{1,k;l}^{\pm},\,\mathsf{X}_{3,s;n}^{\pm}],\,\mathsf{X}_{2,r;m}^{\pm}]\equiv 0
    ~~ \hbox{mod}~~ \mathbf{W}_{l+m+n+1}.
    \end{gather}
   \end{lemma}
\begin{proof}
    Relation \eqref{LM5:eq:1} is equivalent to \eqref{DS:10}, and relation \eqref{LM5:eq:2} follows from relation \eqref{LM5:eq:1} and Lemma \ref{LM-number}. 
    
\end{proof}

Now, we proceed to demonstrate Theorem \ref{main1}. 

\textbf{Proof of Theorem \ref{main1}. }By Lemma \ref{LM-number}--\ref{LM5}, it is clear that $\varphi$ is an epimorphism. We are left to prove the injectivity of $\varphi$. Denote $\mathfrak{B}_Y(\mathfrak{g})$ by the second set mentioned in Proposition \ref{Y:PBW} endowed with an ordering:
\begin{equation*}
    \begin{cases}
        x_{\alpha,\overline{{\bf m}}_1^0}^-\prec h_{i,m_2}\prec x_{\beta,\overline{{\bf m}}_3^0}^+,\quad  &\forall \alpha,\beta\in\mathcal{R}^+,\ m_1,m_2,m_3\in\mathbb{Z}_+, \\
    h_{i_1,m_1}\prec h_{i_2,m_2}, &\text{if}~~i_1<i_2~~\text{or}~~i_1=i_2,\ m_1<m_2, \\
    x_{\alpha,\overline{{\bf m}}_1^0}^{\pm}\prec x_{\beta,\overline{{\bf m}}_2^0}^{\pm}, &\text{if}~~\alpha\prec \beta~~\text{or}~~\alpha=\beta,\ m_1<m_2. 
    \end{cases}
\end{equation*}
Due to Corollary \ref{PBW:Uh}, $\varphi(\mathfrak{B}_Y(\mathfrak{g}))$ is also linearly independent. 

We can now determine a basis of $\operatorname{gr}_{\mathbf{W}}\mathcal{U}_{\hbar}(\mathcal{L}\mathfrak{g})$. It suffices to show a basis of $\mathbf{W}_p/\mathbf{W}_{p+1}$ for each $p\in\mathbb{Z}_+$. Suppose $\mathcal{M}$ is a monomial in $\bar{\mathsf{H}}_{i,0;m}$, $\bar{\mathsf{X}}_{i,0;m}^{\pm}$ with at most one power of odd elements such that $\hbar^{t}\mathcal{M}\in \mathbf{W}_p$. This implies $\mathcal{M}\in \mathbf{W}_{p-t}$. Hence, the set $\bar{\mathfrak{B}}_U^{[p]}(\mathfrak{g})$ of all monomials of the form $\overline{\hbar^{t}\mathcal{M}}$ such that $\mathcal{M}$ is a monomial of degree $p-t$ in $\bar{\mathsf{H}}_{i,0;m}$, $\bar{\mathsf{X}}_{i,0;m}^{\pm}$ with at most one power of odd elements is a basis of $\mathbf{W}_p/\mathbf{W}_{p+1}$. 

Let $\widetilde{\mathfrak{B}}^{[p]}_Y(\mathfrak{g})$ be the set consisting of all monomials $\hbar^t\mathcal{N}$ such that $\mathcal{N}\in \mathfrak{B}_Y(\mathfrak{g})$ has degree $p-t$. It follows that the mapping $\varphi(\widetilde{\mathfrak{B}}^{[p]}_Y(\mathfrak{g}))=\bar{\mathfrak{B}}_U^{[p]}(\mathfrak{g})$ for each $p$. Thus, $\varphi$ is an isomorphism over $\mathbb{C}[\hbar]$. 

\qed

This theorem leads to a series of significant corollaries that precisely characterize structure of the super Yangian $Y(\mathfrak{g})$. 
The following provides a PBW basis for the super Yangian $Y_{\hbar}(\mathfrak{g})$. 
\begin{corollary}\label{Y-PBW-complete}
    The set of all ordered monomials in $x_{\alpha,\overline{\bf m}^{0}}^{\pm}$ and $h_{i,m}$ {\rm (}with the powers of odd elements not greater than 1{\rm )} for $\alpha\in\mathcal{R}^+$, $i\in I$, $m\in\mathbb{Z}_+$ forms a basis for $Y_{\hbar}(\mathfrak{g})$ over $\mathbb{C}[\hbar]$.
\end{corollary}

Based on this basis, we immediately derive an analog of the classical limit for $Y_{\hbar}(\mathfrak{g})$. 
\begin{corollary}
The graded algebra $\operatorname{gr}Y_{\hbar}(\mathfrak{g})$ is isomorphic to $U(\mathfrak{g}[u])[\hbar]$.
\end{corollary}

Combining with \cite[Proposition 3.8]{Na20}, we can deduce
\begin{corollary}\label{Y:center}
    The center of the super Yangian $Y_{\hbar}(\mathfrak{g})$ is trivial. 
\end{corollary}

\subsection{Isomorphisms of super Yangians}
Let $\mathfrak{g}'$ be the Lie superalgebra of type $\mathfrak{osp}_{4|2}$. As described by Musson in \cite[Section 4.2]{Mu12}, when the value of $\lambda$ is $1$, $-2$ or$-\frac{1}{2}$, the Lie superalgebra $\mathfrak{g}=D(2,1;\lambda)$ is isomorphic to $\mathfrak{g}'$. The next proposition establishes a parallel to these isomorphisms in super Yangians. 
\begin{proposition}\label{to:osp}
    There exists an isomorphism of super Yangians $Y_{\hbar}(\mathfrak{g}')$ and $Y_{\hbar}(\mathfrak{g})$ such that
    \begin{enumerate}
        \item[{\rm ($i$)}] $\mathfrak{g}=D_{\lambda}$. For $\lambda=1$, 
    \begin{align*}
            &h_{3,m}\mapsto h_{3,m},\quad x_{3,m}^{\pm}\mapsto x_{3,m}^{\pm}, \\ &h_{j,m}\mapsto h_{j',m},\quad x_{j,m}^{\pm}\mapsto x_{j',m}^{\pm},\quad \{j,j'\}=\{1,2\};
        \end{align*}
        for $\lambda=-2$,
        \begin{align*}
            &h_{3,m}\mapsto \frac{1}{2}h_{3,m},\quad x_{3,m}^{\pm}\mapsto x_{3,m}^{\pm}, \\ &h_{j,m}\mapsto \frac{1}{2}h_{j',m},\quad x_{j,m}^{\pm}\mapsto x_{j',m}^{\pm},\quad \{j,j'\}=\{1,2\};
        \end{align*}
        for $\lambda=-1/2$,
        \begin{align*}
            &h_{3,m}\mapsto 2h_{3,m},\quad x_{3,m}^{\pm}\mapsto x_{3,m}^{\pm}, \\ &h_{j,m}\mapsto 2h_{j',m},\quad x_{j,m}^{\pm}\mapsto x_{j',m}^{\pm},\quad \{j,j'\}=\{1,2\}.
        \end{align*}
        
        \item[{\rm ($ii$)}] $\mathfrak{g}=D_{\lambda}^{\circ}$. For $\lambda=1$,
        \begin{align*}
            &h_{1,m}\mapsto-h_{1,m},\quad x_{1,m}^{\pm}\mapsto x_{1,m}^{\pm}, \\ &h_{j,m}\mapsto-h_{j',m},\quad x_{j,m}^{\pm}\mapsto x_{j',m}^{\pm},\quad \{j,j'\}=\{2,3\};
        \end{align*}
        for $\lambda=-2$, 
        \begin{gather*}
            h_{i,m}\mapsto-h_{i,m},\quad x_{i,m}^{\pm}\mapsto x_{i,m}^{\pm};~~~~~~~~~~~~~~~~~~~~~~~~
        \end{gather*}
        for $\lambda=-1/2$,
        \begin{align*}
            &h_{2,m}\mapsto 2h_{2,m},\quad x_{2,m}^{\pm}\mapsto x_{2,m}^{\pm}, \\ &h_{j,m}\mapsto 2h_{j',m},\quad x_{j,m}^{\pm}\mapsto x_{j',m}^{\pm},\quad \{j,j'\}=\{1,3\}.
        \end{align*}

    \end{enumerate}
\end{proposition}
\begin{proof}
    By comparing the defining relations of the Drinfeld super Yangian $Y_{\hbar}(\mathfrak{g}')$ in \cite[Section 6.2]{FT23} with Definition \ref{Def-Y}, we prove this proposition.
    
\end{proof}

\begin{remark}
    Proposition \ref{to:osp} allows us to replace $\mathfrak{g}$ in Corollary \ref{Y-PBW-complete}--\ref
    {Y:center} by $\mathfrak{g}'$.
\end{remark}

\vspace{1em}
Previous studies \cite{Ka77,Mu12,Se11} have demonstrated that, in contrast to classical Lie algebra, odd reflections can alter  Dynkin diagrams. For any fixed $\lambda$, the reflection $s_{\alpha_2}$ associated with an odd root $\alpha_2$ transforms the Dynkin diagram $G(D_{\lambda})$ into the Dynkin diagram $G(D_{\lambda}^{\circ})$. Furthermore, the corresponding quantum reflection establishes an isomorphism between the quantum affine(loop) superalgebras $U_q(\widehat{D}_{\lambda})$ and $U_q(\widehat{D}_{\lambda}^{\circ})$ \cite{HSTY08}. Under the framework of Section \ref{se:superYangian:fromto}, we propose the following conjecture concerning odd reflections in the context of $Y_{\hbar}(\mathfrak{g})$. 
\begin{conjecture}
    The super Yangians $Y_{\hbar}(\widehat{D}_{\lambda})$ and $Y_{\hbar}(\widehat{D}_{\lambda}^{\circ})$ are  isomorphic. 
\end{conjecture}

\medskip
\section{Hopf superalgebra structure of $Y_{\hbar}(\mathfrak{g})$}\label{se:hopfstructure}
The goal of this section is to establish the Hopf superalgebra structure of the super Yangian $Y_{\hbar}(\mathfrak{g})$. To achieve this, we first construct its minimal realization, which provides a finite set of generators for defining coproduct and antipode. Such work was done by \cite{GNW18} for affine Kac-Moody algebras, \cite{MS22} for $A(m,n)$ and \cite{Ma23} for $B(m,n)$. 

Remark \ref{hbar} allows us to specialize to $\hbar=1$ without loss of generality. By convention, we write $Y_1(\mathfrak{g})$ simply  as $Y(\mathfrak{g})$. 

\subsection{Minimalistic presentation}
\begin{proposition}\label{minimalistic}
    The super Yangian $Y(\mathfrak{g})$ is isomorphic to the superalgebra $\mathring{Y}(\mathfrak{g})$ generated by the finite set
    \begin{gather*}
       \{ h_{i,0},\ h_{i,1},\ x_{i,0}^{\pm},\ x_{i,1}^{\pm}\quad \text{with}\quad i\in I\}
    \end{gather*}
    with the defining relations
    \begin{align} \label{MP:1}
        &[h_{i,m},\,h_{j,n}]=0,\quad m,n=0,1, \\ \label{MP:2}
        &[h_{i,0},\,x_{j,n}^{\pm}]=\pm(\alpha_i,\alpha_j)x_{j,n}^{\pm},\quad n=0,1, \\ \label{MP:3}
        &[\widetilde{h}_{i,1},\,x_{j,0}^{\pm}]=\pm(\alpha_i,\alpha_j)x_{j,1}^{\pm},\quad \widetilde{h}_{i,1}:=h_{i,1}-\frac{1}{2}h_{i,0}^2, \\ \label{MP:4}
        &[x_{i,m}^+,\,x_{j,n}^-]=\delta_{ij}h_{i,m+n},\quad m+n=0,1, \\ \label{MP:5}
        &[x_{i,1}^{\pm},\,x_{j,0}^{\pm}]=[x_{i,0}^{\pm},\,x_{j,1}^{\pm}]\pm\frac{(\alpha_i,\alpha_j)}{2}\{x_{i,0}^{\pm},\,x_{j,0}^{\pm}\}, \\ \label{MP:6}
        &[x_{i,0}^{\pm},\,x_{j,0}^{\pm}]=0,\quad (\alpha_i,\alpha_j)=0, \\ \label{MP:7}
        &[x_{i,0}^{\pm},\,[x_{i,0}^{\pm},\,x_{j,0}^{\pm}]]=0,\quad \mathfrak{g}=D_{\lambda},\ (\alpha_i,\alpha_j)\neq 0,\ \text{and}~~i\neq j, \\ \label{MP:8}
        &(\alpha_1,\alpha_3)[[x_{1,0}^{\pm},\,x_{2,0}^{\pm}],\,x_{3,0}^{\pm}]=(\alpha_1,\alpha_2)[[x_{1,0}^{\pm},\,x_{3,0}^{\pm}],\,x_{2,0}^{\pm}],\quad \mathfrak{g}=D_{\lambda}^{\circ}. 
    \end{align}
    Here, the generators of $\mathring{Y}(\mathfrak{g})$ have $\mathbb{Z}_2$-grading $|x_{i,0}^{\pm}|=|x_{i,1}^{\pm}|=|i|$ and $|h_{i,0}|=|h_{i,1}|=0$. Moreover, define in 
    $\mathring{Y}(\mathfrak{g})$ for $n\geqslant 0$,
    \begin{equation}\label{generating}
    \begin{split}
        &\mathfrak{g}=D_{\lambda}\begin{cases}
            x_{j,n+1}^{\pm}=\pm\frac{1}{b_{ij}}[\widetilde{h}_{i,1},\,x_{j,n}^{\pm}],\quad (i,j)\in\{(1,1),(1,2),(2,3)\}, \\
            h_{i,n+1}=[x_{i,0}^+,\,x_{i,n
            +1}^-].
        \end{cases} \\
        &\mathfrak{g}=D_{\lambda}^{\circ}\begin{cases}
            x_{j,n+1}^{\pm}=\pm\frac{1}{b_{ij}}[\widetilde{h}_{i,1},\,x_{j,n}^{\pm}],\quad (i,j)\in\{(1,2),(2,3),(3,1)\}, \\
            h_{i,n+1}=[x_{i,0}^+,\,x_{i,n+1}^-].
        \end{cases}
    \end{split}
    \end{equation}
    Then the isomorphism assigns to the generators of $Y(\mathfrak{g})$ is given by
    \begin{gather*}
        h_{i,m}\mapsto h_{i,m},\quad x_{i,m}^{\pm}\mapsto x_{i,m}^{\pm}. 
    \end{gather*}
\end{proposition}
\begin{remark}
    The superalgebra $\mathring{Y}(\mathfrak{g})$ is referred to as the \textit{minimalistic presentation}. Here, we employ the same notation for the generators of $\mathring{Y}(\mathfrak{g})$ by convention, as this will not lead to confusion in our paper.
\end{remark}

\begin{proof}
    It suffices to prove that relations \eqref{SY:1}--\eqref{SY:8} can be deduced from relations \eqref{MP:1}--\eqref{MP:8}. The proof proceeds by cases:
    
    (1) $\mathfrak{g}=D_{\lambda}$. By the definition of $\widetilde{h}_{i,1}$, we have for $i,j\in I$, 
    \begin{gather}\label{MP:proof:1}
        [\widetilde{h}_{i,1},\,\widetilde{h}_{j,1}]=0.
    \end{gather}
    Thus, for $j=2$ or $3$, $n\geqslant 1$, 
    \begin{gather}\label{MP:proof:2}
        [\widetilde{h}_{3,1},\,x_{j,n}^{\pm}]=\pm \frac{1}{b_{j-1,j}}[\widetilde{h}_{3,1},\,[\widetilde{h}_{j-1,1},\,x_{j,n-1}^{\pm}]] 
        =\pm \frac{1}{b_{j-1,j}}[\widetilde{h}_{j-1,1},[\widetilde{h}_{3,1},\,x_{j,n-1}^{\pm}]],
    \end{gather}
    which implies that
    \begin{gather*}
        [\widetilde{h}_{3,1},\,x_{j,n}^{\pm}]=\pm b_{3j}x_{j,n+1}^{\pm},
    \end{gather*}
    by relation \eqref{MP:3} and induction on $n$. By Corollary \ref{Y-PBW-complete}, the assignments
    \begin{align*}
        &\varepsilon_1:\ h_{i,m}\mapsto h_{i,m},\quad\quad\quad ~~ x_{i,m}^{\pm}\mapsto x_{i,m}^{\pm}, \\
        &\varepsilon_2:\ h_{i,m}\mapsto \lambda^{-1}h_{4-i,m},\quad x_{i,m}^{\pm}\mapsto \sqrt{\lambda}^{-1}x_{4-i,m}^{\pm},
    \end{align*}
    for $i\in\{1,2\}$ are both embeddings from the super Yangians $Y(\mathfrak{sl}_{2|1})$ to $Y(D_{\lambda})$. Then the relations \eqref{SY:1}--\eqref{SY:8} involving the elements $h_{i,m},x_{i,m}^{\pm}$ in $\mathring{Y}(\mathfrak{g})$ restricted to the index pairs $(1,2)$, $(2,1)$, $(3,2)$ and $(2,3)$ hold by the minimalistic presentation of $Y(\mathfrak{sl}_{2|1})$ proved in \cite[Section 3.3]{MS22}. In addition,  it is easy to see that
    \begin{gather*}
        [h_{1,m},\,h_{3,m}]=[h_{1,m},\,x_{3,n}^{\pm}]=[h_{3,m},\,x_{1,n}^{\pm}]=[x_{1,m}^{\pm},\,x_{3,n}^{\pm}]=0
    \end{gather*}
    for all $m,n\in\mathbb{Z}_+$. 

    (2) $\mathfrak{g}=D_{\lambda}^{\circ}$. In this case, \eqref{MP:proof:1} continues to hold, and with relation \eqref{MP:1}, we see that
    \begin{gather*}
        [h_{i,0},\,x_{j,n}^{\pm}]=\pm b_{ij} x_{j,n}^{\pm},\quad [\widetilde{h}_{i,1},\,x_{j,n}^{\pm}]=\pm b_{ij} x_{j,n+1}^{\pm},\quad \forall i,j\in I,\ n\in\mathbb{Z}_+,
    \end{gather*}
    follows analogously to \eqref{MP:proof:2}. 
    Simple calculations give for $m\in\mathbb{Z}_+$, 
    \begin{align}\label{MP:proof:3}
        &[h_{i,0},\,h_{i,m}]=[[h_{i,0},\,x_{i,0}^+],\,x_{i,m}^-]+[x_{i,0}^+,\,[h_{i,0},\,x_{i,m}^-]]=0, \\ \label{MP:proof:4}
        &[h_{i,1},\,h_{i,m}]=[\widetilde{h}_{i,1},\,h_{i,m}]=[[\widetilde{h}_{i,1},\,x_{i,0}^+],\,x_{i,m}^-]+[x_{i,0}^+,\,[\widetilde{h}_{i,1},\,x_{i,m}^-]]=0, 
    \end{align}
    owing to $b_{ii}=0$. 

    Now we claim that a class of relations 
    \begin{align}\label{MP:inductive:1}
        &[h_{j,m-r},\,h_{i,r}]=0,\quad 0\leqslant r\leqslant m,  \\ \label{MP:inductive:2}
        &[x_{i,m-r}^+,\,x_{i,r}^-]=h_{i,m},\quad 0\leqslant r\leqslant m, \\ \label{MP:inductive:3}
        &[h_{j,m},\,x_{i,n}^{\pm}]=[h_{j,m-1},\,x_{i,n+1}^{\pm}]\pm\frac{b_{ji}}{2}\{h_{j,m-1},\,x_{i,n}^{\pm}\},\quad n\geqslant 0, \\ \label{MP:inductive:4}
        &[x_{j,m}^{\pm},\,x_{i,n}^{\pm}]=[x_{j,m-1}^{\pm},\,x_{i,n+1}^{\pm}]\pm\frac{b_{ji}}{2}\{x_{j,m-1}^{\pm},\,x_{i,n}^{\pm}\},\quad n\geqslant 0, ,\quad n\geqslant 0,
    \end{align}
    for $i\neq j$, $m\geqslant1$. We prove it by induction on $m$. For $m=1$ is clear. Assume that the claim holds for some integer $k\geqslant 1$. A fact is that
    \begin{gather*}
        [h_{j,0},\,h_{i,k+1}]=b_{ji}\left([x_{i,0}^+,\,x_{i,k+1}^-]-[x_{i,0}^+,\,x_{i,k+1}^-]\right)=0. 
    \end{gather*}
    From the inductive hypothesis, we have
    \begin{align*}
        [h_{j,1},\,h_{i,k}]=[\widetilde{h}_{j,1},\,[x_{i,0}^+,\,x_{i,k}^-]]=b_{ji}\left([x_{i,1}^+,\,x_{i,k}^-]-[x_{i,0}^+,\,x_{i,k+1}^-]\right),
    \end{align*}
    and
    \begin{gather*}
        0=\frac{1}{b_{ji}}[h_{j,1},\,h_{i,k}]=\frac{1}{b_{ji}}[\widetilde{h}_{j,1},\,[x_{i,k-r}^+,\,x_{i,r}^-]]=[x_{i,k+1-r}^+,\,x_{i,r}^-]-[x_{i,k-r}^+,\,x_{i,r+1}^-],
    \end{gather*}
    proving relation \eqref{MP:inductive:2} for $m=k+1$. 
    For \eqref{MP:inductive:3} and \eqref{MP:inductive:4} with $m=k+1$, we have for ``$+$'',
    \begin{align*}
        [x_{j,k+1}^+,\,x_{i,n}^+]&=\frac{1}{b_{ij}}[[\widetilde{h}_{i,1},\,x_{j,k}^+],\,x_{i,n}^+]=\frac{1}{b_{ij}}[\widetilde{h}_{i,1},\,[x_{j,k}^+,\,x_{i,n}^+]] \\
        &=\frac{1}{b_{ij}}[\widetilde{h}_{i,1},\,[x_{j,k-1}^+,\,x_{i,n+1}^+]]
        +\frac{1}{2}[\widetilde{h}_{i,1},\,\{x_{j,k-1}^+,\,x_{i,n}^+\}] \\
        &=\frac{1}{b_{ij}}[[\widetilde{h}_{i,1},\,x_{j,k-1}^+],\,x_{i,n+1}^+]
        +\frac{1}{2}([\widetilde{h}_{i,1},\,x_{j,k-1}^+]x_{i,n}^+ + x_{i,n}^+[\widetilde{h}_{i,1},\,x_{j,k-1}^+]) \\
        &=[x_{j,k}^+,\,x_{i,n+1}^+]+\frac{b_{ij}}{2}\{x_{j,k}^+,\,x_{i,n}^+\}, \\
        [h_{j,k+1},\,x_{i,n}^+]&=[[x_{j,k}^+,\,x_{j,1}^-],\,x_{i,n}^+]=[x_{j,k}^+,\,[x_{j,1}^-,\,x_{i,n}^+]]+[[x_{j,k}^+,\,x_{i,n}^+],\,x_{j,1}^-] \\
        &=[[x_{j,k-1}^+,\,x_{i,n+1}^+],\,x_{j,1}^-]+\frac{b_{ji}}{2}[\{x_{j,k-1}^+,\,x_{i,n}^+\},\,x_{j,1}^-] \\
        &=[h_{j,k},\,x_{i,n+1}^+]+\frac{b_{ji}}{2}\{h_{j,k},\,x_{i,n}^+\},
    \end{align*}
    where the equalities
    \begin{gather}\label{MP:proof:5}
        [x_{i,n+1}^+,\,x_{j,0}^-]=[x_{i,n}^+,\,x_{j,1}^-]=\cdots=[x_{i,0}^+,\,x_{j,n+1}^-]=0,
    \end{gather}
    by the action of $\tilde{h}_{j,1}$, $\tilde{h}_{a,1}$($\{i,j,a\}=I$) and induction on $n$. 
    similar to ``$-$''. 
    As for relation \eqref{MP:inductive:1}, the analysis parallel to \cite[Corollary 1.5 \& Lemma 2.1]{Le93} (see also \cite[Lemma 3.16]{MS22} or \cite[Lemma 5.2]{Ma23}) shows that their is a series of elements $\accentset{\approx}{h}_{ij,k}\in\mathring{Y}(\mathfrak{g})$ such that
    \begin{gather*}
        \accentset{\approx}{h}_{ij,0}=h_{i,0},\quad [\accentset{\approx}{h}_{ij,k},\,x_{j,n}^{\pm}]=\pm b_{ij}x_{j,k+n}^{\pm},
    \end{gather*}
    and $\accentset{\approx}{h}_{ij,k}-h_{i,k}$ is a polynomial in $h_{i,0},h_{i,1},\ldots,h_{i,k-1}$. It follows that for $1\leqslant r\leqslant k$
    \begin{gather*}
        [h_{j,k+1-r},\,h_{i,r}]=-[\accentset{\approx}{h}_{ij,r},\,[x_{j,0}^+,\,x_{j,k+1-r}^-]]=-b_{ij}\left([x_{j,r}^+,\,x_{j,k+1-r}^-]-[x_{j,0}^+,\,x_{j,k+1}^-]\right)=0.
    \end{gather*}

    Relation \eqref{SY:1}, \eqref{SY:3}--\eqref{SY:6} for $b_{ij}=0$ can be checked by the same way as \eqref{MP:proof:4} and \eqref{MP:proof:5}. Therefore,  
    it remains to verify relation \eqref{SY:8}. Denote
    \begin{gather*}
        \mathfrak{A}_{l,m,n}^{\pm}:=(1+\lambda)[[x_{1,l}^{\pm},\,x_{2,m}^{\pm}],\,x_{3,n}^{\pm}]+[[x_{1,l}^{\pm},\,x_{3,n}^{\pm}],\,x_{2,m}^{\pm}].
    \end{gather*}
    Relation \eqref{MP:8} is equivalent to $\mathfrak{A}_{0,0,0}^{\pm}=0$. 
    Applying $[\tilde{h}_{i,1},\,\cdot\,]$ for $i=1,2,3$ to  $\mathfrak{A}_{l,m,n}^{\pm}=0$ respectively, we obtain a system of linear equations
    \begin{equation*}
        \begin{cases}
            \pm b_{12}\mathfrak{A}_{l,m+1,n}^{\pm}\pm b_{13}\mathfrak{A}_{l,m,n+1}^{\pm}=0, \\
            \pm b_{21}\mathfrak{A}_{l+1,m,n}^{\pm}\pm b_{23}\mathfrak{A}_{l,m,n+1}^{\pm}=0, \\
            \pm b_{31}\mathfrak{A}_{l+1,m,n}^{\pm}\pm b_{32}\mathfrak{A}_{l,m+1,n}^{\pm}=0, \\
        \end{cases}
    \end{equation*}
    whose coefficient matrix coincides with $\pm A(D_{\lambda}^{\circ})$. Since the determinant of $A(D_{\lambda}^{\circ})$ is nonzero, this immediately implies that
    \begin{gather*}
        \mathfrak{A}_{l+1,m,n}^{\pm}=\mathfrak{A}_{l,m+1,n}^{\pm}
        =\mathfrak{A}_{l,m,n+1}^{\pm}=0.
    \end{gather*}
    Then we get \eqref{SY:8} by induction on $l,m,n$, respectively.

\end{proof}

\subsection{Coproduct, counit and antipode}\label{se:hopfstructure:CCA}
Let $\square$ be the operator defined on the super Yangian $Y(\mathfrak{g})$ by $$\square(x)=x\otimes 1+1\otimes x\quad \text{for}\quad x\in Y(\mathfrak{g}).$$ Although $\square$ is not a superalgebraic homomorphism, it still satisfies the compatibility condition $$\square([x,\,y])=[\square(x),\,\square(y)]\quad \text{for}\quad x,y\in Y(\mathfrak{g}).$$
It prompts us to define an explicit superalgebraic homomorphism $\triangle$ from $Y(\mathfrak{g})$ to the tensor product $Y(\mathfrak{g})\otimes Y(\mathfrak{g})$.

Observe that the assignment 
\begin{gather*}
    h_i\mapsto d_i^{-1}h_{i,0},\quad e_i\mapsto x_{i,0}^+,\quad f_i\mapsto d_i^{-1}x_{i,0}^-
\end{gather*}
naturally defines a superalgebraic embedding
$\iota:\ U(\mathfrak{g})\rightarrow Y(\mathfrak{g}).$
Let $\widehat{\Omega}^+$ denote the image of the half Casimir element $\Omega^+$ under $\iota$. We further define
\begin{gather*}
    \widetilde{x}_{\alpha}^+=\iota(\tilde{e}_{\alpha})\quad \text{and}\quad \widetilde{x}_{\alpha}^-=\iota(\tilde{f}_{\alpha})\quad \text{for}~~\alpha\in\mathcal{R}^+. 
\end{gather*}
The pullback of the bilinear form from $\mathfrak{g}$ yields a non-degenerate invariant $\mathbb{C}$-value bilinear form $(\,\cdot\,,\,\cdot\,)$ on $\iota(\mathfrak{g})$ uniquely characterized by $(\iota(x),\iota(y))=(x,y)$ for all $x,y\in\mathfrak{g}$. 
These constructions immediately imply that$(\widetilde{x}_{\alpha}^+,\widetilde{x}_{\beta}^-)=\delta_{\alpha\beta}$ for $\alpha,\beta\in\mathcal{R}^+$. Combining Lemma \ref{Half-casimir} with the preceding discussion, we arrive at the following lemma.
\begin{lemma}\label{half-casimir-Y}
    The following equations hold in $Y(\mathfrak{g})$,
    \begin{align}\label{half-casimir-Y1}
        &[\square(h_{i,0}),\,\widehat{\Omega}^+]=0, \\ \label{half-casimir-Y2}
        &[\square(x_{i,0}^+),\,\widehat{\Omega}^+]=-x_{i,0}^+\otimes h_{i,0}, \\ \label{half-casimir-Y3}
        &[\square(x_{i,0}^-),\,\widehat{\Omega}^+]=h_{i,0}\otimes x_{i,0}^-.
    \end{align}
\end{lemma}

Lemma \ref{half-casimir-Y} enables us to establish that
\begin{theorem}\label{coproduct}
    Let $\triangle:\ Y(\mathfrak{g})\rightarrow Y(\mathfrak{g})\otimes Y(\mathfrak{g})$ be the linear $\mathbb{Z}_2$-graded even operator satisfying
\begin{align*}
    \triangle(h_{i,0})&=\square(h_{i,0}),\quad \triangle(x_{i,0}^{\pm})=\square(x_{i,0}^{\pm}), \\
    \triangle(h_{i,1})&=\square(h_{i,1})+h_{i,0}\otimes h_{i,0}+[h_{i,0}\otimes 1,\,\widehat{\Omega}^+] \\
    &=\square(h_{i,1})+h_{i,0}\otimes h_{i,0}-\sum_{\alpha\in\mathcal{R}^+}(\alpha_i,\alpha)\widetilde{x}_{\alpha}^-\otimes \widetilde{x}_{\alpha}^+.
\end{align*} 
Then the operator $\triangle$ defines a superalgebraic homomorphism. Furthermore, $\triangle$ forms a coproduct on the super Yangian $Y(\mathfrak{g})$.
\end{theorem}
Within the framework of the recurrence relation \eqref{generating}, the action of $\triangle$ can be extended to the entire generating set $\{h_{i,m},x_{i,m}^{\pm}\}_{i\in I,m\in\mathbb{Z}_+}$ of $Y(\mathfrak{g})$ through the following construction. It follows immediately from the definition of $\tilde{h}_{i,1}$
that $\triangle(\widetilde{h}_{i,1})=\square(\widetilde{h}_{i,1})+[h_{i,0}\otimes 1,\,\widehat{\Omega}^+]$. Then we have for $n\geqslant 0$,
\begin{equation}\label{generating-tri}
    \begin{split}
        &\mathfrak{g}=D_{\lambda}\begin{cases}
            \triangle(x_{j,n+1}^{\pm})=\pm\frac{1}{b_{ij}}[\triangle(\widetilde{h}_{i,1}),\,\triangle(x_{j,n}^{\pm})],\quad (i,j)\in\{(1,1),(1,2),(2,3)\}, \\
            \triangle(h_{i,n+1})=[\triangle(x_{i,0}^+),\,\triangle(x_{i,n+1}^-)].
        \end{cases} \\
        &\mathfrak{g}=D_{\lambda}^{\circ}\begin{cases}
            \triangle(x_{j,n+1}^{\pm})=\pm\frac{1}{b_{ij}}[\triangle(\widetilde{h}_{i,1}),\,\triangle(x_{j,n}^{\pm})],\quad (i,j)\in\{(1,2),(2,3),(3,1)\}, \\
            \triangle(h_{i,n+1})=[\triangle(x_{i,0}^+),\,\triangle(x_{i,n+1}^-)].
        \end{cases}
    \end{split}
    \end{equation}
We defer the proof of Theorem \ref{coproduct} to Section \ref{se:hopfstructure:proof}.

\vspace{1em}
Let us define:
\begin{itemize}
    \item the counit $\epsilon:\ Y(\mathfrak{g})\rightarrow \mathbb{C}$ as the superalgebraic homomorphism specified by
    \begin{gather*}
    \epsilon(1)=1,\quad \epsilon(h_{i,m})=\epsilon(x_{i,m}^{\pm})=0;
    \end{gather*}

    \item the antipode $\mathcal{S}:\ Y(\mathfrak{g})\rightarrow Y(\mathfrak{g})$ as the anti-automorphism determined by
    \begin{align*}
    &\mathcal{S}(h_{i,0})=-h_{i,0},\quad \mathcal{S}(x_{i,0}^{\pm})=-x_{i,0}^{\pm}, \\
    &\mathcal{S}(h_{i,1})=-h_{i,1}+h_{i,0}^2-\sum_{\alpha\in\mathcal{R}^+}(\alpha_i,\alpha)(-1)^{|\alpha|}\tilde{x}_{\alpha}^-\tilde{x}_{\alpha}^+. 
    \end{align*}
\end{itemize}
These definitions ensure the validity of relation 
\begin{gather*}
    \mu\circ (\mathcal{S}\otimes 1)\circ \triangle =\mu\circ (1\otimes \mathcal{S})\circ \triangle =\nu\circ \epsilon,
\end{gather*}
where $\mu$ is the multiplication map mentioned in Section \ref{se:simpleLiesuper:HC} and $\nu: \mathbb{C}\rightarrow Y(\mathfrak{g})$ is the embedding $\nu(a)=a1$ for all $a\in\mathbb{C}$.

\begin{corollary}
    The super Yangian $Y(\mathfrak{g})$ admits a Hopf superalgebra structure endowed with coproduct $\triangle$, counit $\epsilon$ and antipode $\mathcal{S}$. 
\end{corollary}

\subsection{Proof of Theorem \ref{coproduct}}\label{se:hopfstructure:proof}
We split the proof into three parts. 
We begin by evaluating $\triangle(x_{i,1}^{\pm})$, establishing the explicit formula
\begin{align}\label{coproduct-1}
    \triangle(x_{i,1}^+)&=\square(x_{i,1}^+)-[1\otimes x_{i,0}^+,\,\widehat{\Omega}^+], \\ \label{coproduct-2}
    \triangle(x_{i,1}^-)&=\square(x_{i,1}^-)+[x_{i,0}^-\otimes 1,\,\widehat{\Omega}^+].
\end{align}
By Lemma \ref{half-casimir-Y}, we have for a suitable pair $(i,j)$ as in \eqref{generating-tri}, 
\begin{align*}
    \triangle(x_{i,1}^+)&=\frac{1}{b_{ji}}[\triangle(\widetilde{h}_{j,1}),\,\triangle(x_{i,0}^+)] 
    =\frac{1}{b_{ji}}\left([\square(\widetilde{h}_{j,1}),\,\square(x_{i,0}^+)]+[[h_{j,0}\otimes 1,\,\widehat{\Omega}^+],\,\square(x_{i,0}^+)]\right) \\
    &=\square(x_{i,1}^+)-\frac{1}{b_{ji}}\left([1\otimes h_{j,0},\,[\widehat{\Omega}^+,\,\square(x_{i,0}^+)]]+[[1\otimes h_{j,0}^+,\,\square(x_{i,0}^+)],\,\widehat{\Omega}^+]\right) \\
    &=\square(x_{i,1}^+)-\frac{1}{b_{ji}}[1\otimes h_{j,0}^+,\,x_{i,0}^+\otimes h_{i,0}]-[1\otimes x_{i,0}^+,\,\widehat{\Omega}^+] 
    =\square(x_{i,1}^+)-[1\otimes x_{i,0}^+,\,\widehat{\Omega}^+].
\end{align*}
Similar to $\triangle(x_{i,1}^-)$.

\vspace{1em}
Next, we verify that $\triangle$ preserves the defining relations of $Y(\mathfrak{g})$. By Proposition \ref{minimalistic}, we only need to check relations \eqref{MP:1}--\eqref{MP:5} involving $h_{i,1}$, $\widetilde{h}_{i,1}$ and $x_{i,1}^{\pm}$. 

Relation \eqref{MP:2}, \eqref{MP:3} is trivial. For relation \eqref{MP:4}, one gets
\begin{align*}
    [\triangle(x_{i,0}^+),\,\triangle(x_{j,1}^-)]&=[\square(x_{i,0}^+),\,\square(x_{j,1}^-)]-[\square(x_{i,0}^+),\,[\widehat{\Omega}^+,\,x_{j,0}^-\otimes 1]] \\
    &=\delta_{ij}\square(h_{i,1})-[[\square(x_{i,0}^+),\,\widehat{\Omega}^+],\,x_{j,0}^-\otimes 1]-[\widehat{\Omega}^+,\,[\square(x_{i,0}^+),\,x_{j,0}^-\otimes 1]] \\
    &=\delta_{ij}\square(h_{i,1})+[x_{i,0}^+\otimes h_{i,0},\,x_{j,0}^-\otimes 1]-[\widehat{\Omega}^+,\,\delta_{ij}h_{i,0}\otimes 1] \\
    &=\delta_{ij}\triangle(h_{i,1}). 
\end{align*}
Similarly, we also have
\begin{align*}
    [\triangle(x_{i,1}^+),\,\triangle(x_{j,0}^-)]=\delta_{ij}\square(h_{i,1})+[1\otimes x_{i,0}^+,\,h_{j,0}\otimes x_{j,0}^-]-[1\otimes \delta_{ij}h_{i,0},\,\widehat{\Omega}^+]=\delta_{ij}\triangle(h_{i,1}).
\end{align*}

We are left to calculate \eqref{MP:5}. We only prove for ``$+$'' and the case of ``$-$'' can be checked in a similar way. Indeed, 
\begin{align*}
    &\quad\ [\triangle(x_{i,1}^+),\,\triangle(x_{j,0}^+)]-[\triangle(x_{i,0}^+),\,\triangle(x_{j,1}^+)] \\
    &=[\square(x_{i,1}^+),\,\square(x_{j,0}^+)]-[[1\otimes x_{i,0}^+,\,\widehat{\Omega}^+],\,\square(x_{j,0}^+)]-[\square(x_{i,0}^+),\,\square(x_{j,1}^+)]+[\square(x_{i,0}^+),\,[1\otimes x_{j,0}^+,\,\widehat{\Omega}^+]] \\
    &=\frac{b_{ij}}{2}\square(\{x_{i,0}^+,\,x_{j,0}^+\})-[1\otimes x_{i,0}^+,\,x_{j,0}^+\otimes h_{j,0}]+[x_{i,0}^+\otimes h_{i,0},\,1\otimes x_{j,0}^+] \\
    &=\frac{b_{ij}}{2}\square(\{x_{i,0}^+,\,x_{j,0}^+\})+b_{ij}\left((-1)^{|i||j|}x_{j,0}^+\otimes x_{i,0}^++x_{i,0}^+\otimes x_{j,0}^+\right)=\frac{b_{ij}}{2}\{\square(x_{i,0}^+),\,\square(x_{j,0}^+)\}. 
\end{align*}

The rest is to prove that the operator $\triangle$ preserves $[h_{i,1},h_{j,1}]=0$. It is enough to show that 
$$[\triangle(\widetilde{h}_{i,1}),\,\triangle(\widetilde{h}_{j,1})]=0, $$
or equivalently,
\begin{gather}\label{long-eq}
    [[h_{i,0}\otimes 1,\,\widehat{\Omega}^+],\,[h_{j,0}\otimes 1,\,\widehat{\Omega}^+]]=-[\square(h_{i,1}),\,[h_{j,0}\otimes 1,\,\widehat{\Omega}^+]]-[[h_{i,0}\otimes 1,\,\widehat{\Omega}^+],\,\square(h_{j,1})].
\end{gather}
To establish \eqref{long-eq}, we analyze the left-hand side (LHS) and right-hand side (RHS) separately. The LHS expands to
\begin{equation*}
    \frac{1}{2}\sum_{p\geqslant 1}\sum_{{\scriptsize \begin{array}{c}
        \operatorname{ht}(\alpha+\beta)=p   \\
          \alpha,\beta\in\mathcal{R}^+ 
    \end{array}} }(\alpha_i,\alpha)(\alpha_j,\beta)\left([\widetilde{x}_{\alpha}^-,\,\widetilde{x}_{\beta}^-]\otimes \{\widetilde{x}_{\alpha}^+,\,\widetilde{x}_{\beta}^+\}+\{\widetilde{x}_{\alpha}^-,\,\widetilde{x}_{\beta}^-\}\otimes [\widetilde{x}_{\alpha}^+,\,\widetilde{x}_{\beta}^+]\right).
\end{equation*}
The RHS of \eqref{long-eq} is equal to
\begin{gather*}
    \sum_{\alpha\in\mathcal{R}^+}(\alpha_j,\alpha)[\square(\widetilde{h}_{i,1}),\,\widetilde{x}_{\alpha}^-\otimes \widetilde{x}_{\alpha}^+]-\sum_{\alpha\in\mathcal{R}^+}(\alpha_i,\alpha)[\square(\widetilde{h}_{j,1}),\,\widetilde{x}_{\alpha}^-\otimes \widetilde{x}_{\alpha}^+].
\end{gather*}
Therefore, equality \eqref{long-eq} holds if equations \eqref{long:eq:a1} and \eqref{long:eq:a2} are satisfied for all $p\geqslant 1$,
\begin{align}\label{long:eq:a1}
    \sum_{\operatorname{ht}(\alpha+\beta)=p}(\alpha_i,\alpha)(\alpha_j,\beta)\{\widetilde{x}_{\alpha}^-,\,\widetilde{x}_{\beta}^-\}\otimes [\widetilde{x}_{\alpha}^+,\,\widetilde{x}_{\beta}^+]&=2\sum_{\operatorname{ht}(\alpha)=p}[\check{h}_{ij,\alpha},\,\widetilde{x}_{\alpha}^-]\otimes \widetilde{x}_{\alpha}^+, \\ \label{long:eq:a2}
    \sum_{\operatorname{ht}(\alpha+\beta)=p}(\alpha_i,\alpha)(\alpha_j,\beta)[\widetilde{x}_{\alpha}^-,\,\widetilde{x}_{\beta}^-]\otimes 
    \{\widetilde{x}_{\alpha}^+,\,\widetilde{x}_{\beta}^+\}
    &=2\sum_{\operatorname{ht}(\alpha)=p}\widetilde{x}_{\alpha}^-\otimes [\check{h}_{ij,\alpha},\,\widetilde{x}_{\alpha}^+],
\end{align}
where $\check{h}_{ij,\alpha}=(\alpha_j,\alpha)\widetilde{h}_{i,1}-(\alpha_i,\alpha)\square(\widetilde{h}_{j,1})$ and all the roots $\alpha,\beta$ considered lie in $\mathcal{R}^+$.

We proceed to introduce a function on the sub-superalgebra $\iota(\mathfrak{g})$. Define
\begin{align*}
    &J(h_{i,0})=\widetilde{h}_{i,1}+\widetilde{v}_{i},\quad \text{for}~~\widetilde{v}_i=\frac{1}{2}\sum_{\alpha\in\mathcal{R}^+}(\alpha_i,\alpha)\widetilde{x}_{\alpha}^-\widetilde{x}_{\alpha}^+, \\
    &J([x,\,y])=[x,\,J(y)],\quad \text{for all }~y\in\iota(\mathfrak{h}),\ x\in\iota(\mathfrak{g}). 
\end{align*}
The well-definedness follows from \cite{GNW18,MS22}. 
Then we have 
\begin{gather*}
    [h_{k,0},\,J(\widetilde{x}_{\alpha}^{\pm})]=-[\widetilde{x}_{\alpha}^{\pm},\,J(h_{k,0})]=\pm(\alpha_k,\alpha)J(\widetilde{x}_{\alpha}^{\pm}).
\end{gather*}
This implies that
\begin{gather*}
    [\check{h}_{ij,\alpha},\,\widetilde{x}_{\alpha}^-]\otimes \widetilde{x}_{\alpha}^+=[\check{v}_{ji,\alpha},\,\widetilde{x}_{\alpha}^-]\otimes \widetilde{x}_{\alpha}^+, 
\end{gather*}
where $\check{v}_{ij,\alpha}=(\alpha_j,\alpha)\widetilde{v}_i-(\alpha_i,\alpha)\widetilde{v}_j$.
By the definition of $\widetilde{v}_i$, we have
\begin{gather}\label{long:eq:a3}
    [\check{v}_{ji,\alpha},\,\widetilde{x}_{\alpha}^-]=-\frac{1}{2}\sum_{\beta\in\mathcal{R}^+}C_{i,j;\alpha,\beta}\left((-1)^{|\alpha||\beta|}[\widetilde{x}_{\beta}^-,\,\widetilde{x}_{\alpha}^-]\widetilde{x}_{\beta}^++\widetilde{x}_{\beta}^-[\widetilde{x}_{\beta}^+,\,\widetilde{x}_{\alpha}^-]\right),
\end{gather}
where $C_{i,j;\alpha,\beta}=(\alpha_i,\alpha)(\alpha_j,\beta)-(\alpha_j,\alpha)(\alpha_i,\beta)$. 
Substitute 
\begin{align*}
    &[\widetilde{x}_{\beta}^-,\,\widetilde{x}_{\alpha}^-]=(\widetilde{x}_{\gamma}^+,[\widetilde{x}_{\beta}^-,\,\widetilde{x}_{\alpha}^-])\widetilde{x}_{\gamma}^-,\quad \text{for }~\gamma=\alpha+\beta\in\mathcal{R}^+, \\
    &[\widetilde{x}_{\beta}^+,\,\widetilde{x}_{\alpha}^-]=([\widetilde{x}_{\beta}^+,\,\widetilde{x}_{\alpha}^-],\widetilde{x}_{\gamma}^-)\widetilde{x}_{\gamma}^+,\quad \text{for }~\gamma=\beta-\alpha\in\mathcal{R}^+, \\
    &[\widetilde{x}_{\beta}^+,\,\widetilde{x}_{\alpha}^-]=(\widetilde{x}_{\gamma}^+,[\widetilde{x}_{\beta}^+,\,\widetilde{x}_{\alpha}^-])\widetilde{x}_{\gamma}^-,\quad \text{for }~\gamma=\alpha-\beta\in\mathcal{R}^+
\end{align*}
into \eqref{long:eq:a3}, we obtain
\begin{equation*}
    \begin{split}
        -2[\check{v}_{ji,\alpha},\,\widetilde{x}_{\alpha}^-]
        &=\sum_{\gamma+\beta=\alpha}C_{i,j;\beta+\gamma,\beta}([\widetilde{x}_{\gamma}^+,\widetilde{x}_{\beta}^+],\,\widetilde{x}_{\alpha}^-])\widetilde{x}_{\beta}^-\widetilde{x}_{\gamma}^- \\
        &+\sum_{\gamma-\beta=\alpha}C_{i,j;\gamma-\beta,\beta}(\widetilde{x}_{\gamma}^+,[\widetilde{x}_{\alpha}^-,\,\widetilde{x}_{\beta}^-])\widetilde{x}_{\gamma}^-\widetilde{x}_{\beta}^+ 
        +\sum_{\beta-\gamma=\alpha}C_{i,j;\beta-\gamma,\beta}([\widetilde{x}_{\beta}^+,\,\widetilde{x}_{\alpha}^-],\widetilde{x}_{\gamma}^-)\widetilde{x}_{\beta}^-\widetilde{x}_{\gamma}^+.
    \end{split}
\end{equation*}
The sum of the second and third terms in the right-hand side of this equation vanish,  
which forces the right-hand side of \eqref{long:eq:a1} equal to
\begin{align*}
    &\quad\,-\sum_{\operatorname{ht}(\alpha)=p}\sum_{\gamma+\beta=\alpha}C_{i,j;\beta+\gamma,\beta}([\widetilde{x}_{\gamma}^+,\widetilde{x}_{\beta}^+],\,\widetilde{x}_{\alpha}^-])\widetilde{x}_{\beta}^-\widetilde{x}_{\gamma}^- \otimes \widetilde{x}_{\alpha}^+ \\
    &=-\sum_{\operatorname{ht}(\alpha+\beta)=p}C_{i,j;\alpha,\beta}\widetilde{x}_{\beta}^-\widetilde{x}_{\alpha}^- \otimes ([\widetilde{x}_{\alpha}^+,\widetilde{x}_{\beta}^+],\,\widetilde{x}_{\alpha+\beta}^-])\widetilde{x}_{\alpha+\beta}^+ \\
    &=\sum_{\operatorname{ht}(\alpha+\beta)=p}(\alpha_j,\alpha)(\alpha_i,\beta)\widetilde{x}_{\beta}^-\widetilde{x}_{\alpha}^- \otimes [\widetilde{x}_{\alpha}^+,\widetilde{x}_{\beta}^+]+\sum_{\operatorname{ht}(\alpha+\beta)=p}(\alpha_i,\alpha)(\alpha_j,\beta)(-1)^{|\alpha||\beta|}\widetilde{x}_{\beta}^-\widetilde{x}_{\alpha}^- \otimes [\widetilde{x}_{\beta}^+,\widetilde{x}_{\alpha}^+] \\
    &=\mathcal{N}(p),
\end{align*}
where $\mathcal{N}(p)$ denotes the left-hand side of \eqref{long:eq:a1} subject to $p\geqslant 1$. So we prove \eqref{long:eq:a1}. A similar method can be used to prove \eqref{long:eq:a2}. 

\vspace{1em}
Finally, we show that the coassociativity of the operator $\triangle$, that is, $\triangle$ satisfies
\begin{gather*}
    (1\otimes \triangle)\circ \triangle=(\triangle\otimes 1)\circ \triangle.
\end{gather*}
We only need to check for $\widetilde{h}_{i,1}$. In fact, 
\begin{align*}
    (1\otimes \triangle)\circ \triangle(\widetilde{h}_{i,1})&=(1\otimes \triangle)\left(\square(\widetilde{h}_{i,1})-\sum_{\alpha\in\mathcal{R}^+}(\alpha_i,\alpha)\widetilde{x}_{\alpha}^-\otimes \widetilde{x}_{\alpha}^+\right) \\
    &=\widetilde{h}_{i,1}\otimes 1\otimes 1+1\otimes \widetilde{h}_{i,1}\otimes 1+1\otimes 1\otimes \widetilde{h}_{i,1} \\
    &\qquad-\sum_{\alpha\in\mathcal{R}^+}(\alpha_i,\alpha)\left(1\otimes \widetilde{x}_{\alpha}^-\otimes \widetilde{x}_{\alpha}^+ +\widetilde{x}_{\alpha}^-\otimes \widetilde{x}_{\alpha}^+\otimes 1+\widetilde{x}_{\alpha}^-\otimes 1\otimes \widetilde{x}_{\alpha}^+\right) \\
    &=(\triangle\otimes 1)\circ \triangle(\widetilde{h}_{i,1}).
\end{align*}

Hence, we complete the proof.

\vspace{2em}
\section*{CRediT authorship contribution statement}
{\bf Hongda Lin:} Writing – original draft. {\bf Honglian Zhang:} Writing – original draft.

\medskip
\section*{Declaration of competing interest}
The authors declare that they have no known competing financial interests or personal relationships that 
could have appeared to influence the work reported in this paper. 

\medskip
\section*{Acknowledgments}
H. Zhang is partially supported by the support of the National Natural Science Foundation of China 12271332 and Natural Science Foundation of Shanghai 22ZR1424600.


\medskip
\begin{thebibliography}{99}
\bibitem{Be94}
J. Beck, Braid group action and quantum affine algebras, Comm. Math. Phys. 165 (3) (1994) 555 -- 568. \href{https://doi.org/10.1007/BF02099423}{https://doi.org/10.1007/BF02099423}. 

\bibitem{BEl02}
G. Benkart, A. Elduque, The Tits construction and the exceptional simple classical Lie superalgebras, Q. J. Math. 54 (2) (2003) 123 -- 137. \href{https://doi.org/10.1093/qmath/hag014}{https://doi.org/10.1093/qmath/hag014}.

\bibitem{BGZ90}
A.J. Bracken, M.D. Gould, R.B. Zhang, Quantum supergroups and solutions of the Yang-Baxter equation, Mod. Phys. Lett. A 5 (11) (1990) 831 -- 840. \href{https://doi.org/10.1142/S0217732390000925}{https://doi.org/10.1142/S0217732390000925}. 

\bibitem{CW25}
Z. Chang, Y. Wang, A Drinfeld Presentation of the Queer Super-Yangian, Preprint, 2025. \href{https://doi.org/10.48550/arXiv.2503.03181}{arXiv:2503.03181}.

\bibitem{CP95}
V. Chari, A. Pressley, A guide to quantum groups, Cambridge University Press, Cambridge, 1995. 

\bibitem{CG15}
P. Conner, N. Guay, 2015. From twisted quantum loop algebras to twisted Yangians. SIGMA 11, 040. \href{https://doi.org/10.3842/SIGMA.2015.040}{https://doi.org/10.3842/SIGMA.2015.040}.

\bibitem{Da12}
I. Damiani, Drinfeld realization of affine quantum algebras: the relations, Publ. Res. Inst. Math. Sci. 48 (3) (2012) 661 -- 733. \href{https://doi.org/10.2977/PRIMS/86}{https://doi.org/10.2977/PRIMS/86}.

\bibitem{DF04}
E. D’Hoker, D.Z. Freedman, Supersymmetric gauge theories and the AdS/CFT correspondence, In: Strings, branes and extra Dimensions, TASI, 2001, pp. 3 -- 159. 

\bibitem{Dr85}
V.G. Drinfeld, Hopf algebras and the quantum Yang-Baxter equation, Dokl. Akad. Nauk SSSR 283 (5) (1985) 1060 -- 1064. 

\bibitem{Dr86}
V.G. Drinfeld, Quantum groups, in: Proceedings of the International Congress of Mathematicians, Vol. 1, 2, Berkeley, Calif., 1986, Amer. Math. Soc., Providence, RI, 1987, pp. 798 -- 820.

\bibitem{Dr87}
V.G. Drinfeld, A new realization of Yangians and of quantum affine algebras, Dokl. Akad. Nauk SSSR 296 (1) (1987) 13 -- 17. 

\bibitem{FT23}
R. Frassek, A. Tsymbaliuk, Orthosymplectic Yangians, Preprint, 2023. \href{https://doi.org/10.48550/arXiv.2311.18818}{arXiv:2311.18818}.

\bibitem{GT13}
S. Gautam, V. Toledano Laredo, Yangians and quantum loop algebras, Sel. Math. New Ser. 19 (2) (2013) 271 -- 336. \href{https://doi.org/10.1007/s00029-012-0114-2}{https://doi.org/10.1007/s00029-012-0114-2}. 

\bibitem{GT16}
S. Gautam, V. Toledano Laredo, Yangians, quantum loop algebras, and abelian difference equations, J. Amer. Math. Soc. 29 (3) (2016) 775 -- 824. \href{https://doi.org/10.1090/jams/851}{https://doi.org/10.1090/jams/851}. 

\bibitem{Go07}
L. Gow, Gauss decomposition of the Yangian $Y(\mathfrak{gl}_{m|n})$, Comm. Math. Phys. 276 (3) (2007) 799 -- 825. \href{https://doi.org/10.1007/s00220-007-0349-5}{https://doi.org/10.1007/s00220-007-0349-5}. 

\bibitem{GM12}
N. Guay, X. Ma, From quantum loop algebras to Yangians, J. Lond. Math. Soc. 86 (3) (2012) 683 -- 700. \href{ https://doi.org/10.1112/jlms/jds021}{https://doi.org/10.1112/jlms/jds021}.

\bibitem{GNW18}
N. Guay, H. Nakajima, C. Wendlandt, Coproduct for Yangians of affine Kac-Moody algebras, Adv. Math. 338 (2018) 865 -- 911. \href{https://doi.org/10.1016/j.aim.2018.09.013}{https://doi.org/10.1016/j.aim.2018.09.013}. 


\bibitem{HSTY08}
I. Heckenberger, F. Spill, A. Torrielli, H. Yamane, Drinfeld second realization of the quantum affine superalgebras of $D^{(1)}(2,1;x)$ via the Weyl groupoid, RIMS K{\^ o}ky{\^ u}roku Bessatsu 8 (2008) 171 -- 216.

\bibitem{Ji85}
M. Jimbo, A $q$-difference analogue of $U(\mathfrak{g})$ and the Yang-Baxter equation, Lett. Math. Phys. 10 (1) (1985), 63 -- 69. \href{https://doi.org/10.1007/BF00704588}{https://doi.org/10.1007/BF00704588}. 

\bibitem{JZ16}
N. Jing, H. Zhang, 2016. Drinfeld realization of quantum twisted affine algebras via braid group. Adv. Math. Phys. 4843075. \href{https://doi.org/10.1155/2016/4843075}{https://doi.org/10.1155/2016/4843075}. 

\bibitem{Ka77}
V.G. Kac, Lie superalgebras, Adv. Math. 26 (1) (1977) 8 -- 96. \href{https://doi.org/10.1016/0001-8708(77)90017-2}{https://doi.org/10.1016/0001-8708(77)90017-2}.

\bibitem{Ka90}
V.G. Kac, Infinite-dimensional Lie algebras, Cambridge University Press, Cambridge, 1990. \href{https://doi.org/10.1017/CBO9780511626234}{https://doi.org/10.1017/CBO9780511626234}.

\bibitem{Le93}
S.Z. Levendorski{\u \i}, On generators and defining relations of Yangians, J. Geom. Phys. 12 (1) (1993) 1 -- 11. \href{https://doi.org/10.1016/0393-0440(93)90084-R}{https://doi.org/10.1016/0393-0440(93)90084-R}. 

\bibitem{LinWZ24}
H. Lin, Y. Wang, H. Zhang, From quantum loop superalgebras to super Yangians, J. Algebra 650 (2024) 299 -- 334. \href{https://doi.org/10.1016/j.jalgebra.2024.03.028}{https://doi.org/10.1016/j.jalgebra.2024.03.028}. 

\bibitem{LinYZ24}
H. Lin, H. Yamane, H. Zhang, 2024. On generators and defining relations of quantum affine superalgebra $U_q(\widehat{\mathfrak{sl}}_{m|n})$. J. Algebra Appl. 23 (1), 2450021. \href{https://doi.org/10.1142/S021949882450021X}{https://doi.org/10.1142/S021949882450021X}. 

\bibitem{LWZ25}
K. Lu, W. Wang, W. Zhang,  Affine $\imath$quantum groups and twisted Yangians in Drinfeld presentations. Comm. Math. Phys. 406 (5), (2025) 98. \href{https://doi.org/10.1007/s00220-025-05263-z}{https://doi.org/10.1007/s00220-025-05263-z}. 

\bibitem{LZ24}
K. Lu, W. Zhang, A Drinfeld type presentation of twisted Yangians of quasi-split type, Preprint, 2024. \href{https://doi.org/10.48550/arXiv.2408.06981}{arXiv:2408.06981}.

\bibitem{Mal99}
J. Maldacena, The large-$N$ limit of superconformal field theories and supergravity, Int. J. Theor. Phys. 38 (4), (1999), 1113 -- 1133. \href{https://doi.org/10.1023/A:1026654312961}{https://doi.org/10.1023/A:1026654312961}. 

\bibitem{Ma23}
A. Mazurenko, Hopf superalgebra structure for Drinfeld super Yangian of Lie superalgebra $B(m,n)$, Preprint, 2023. \href{https://doi.org/10.48550/arXiv.2309.09735}{arXiv:2309.09735}.

\bibitem{MS22}
A. Mazurenko, V.A. Stukopin, Classification of Hopf superalgebra structures on Drinfeld super Yangians, Preprint, 2022. \href{https://doi.org/10.48550/arXiv.2210.08365}{arXiv:2210.08365}. 

\bibitem{Mo24}
A.I. Molev, A Drinfeld-type presentation of the orthosymplectic Yangians, Algebr. Represent. Theory 27 (1) (2024) 469 -- 494. \href{https://doi.org/10.1007/s10468-023-10227-9}{https://doi.org/10.1007/s10468-023-10227-9}. 

\bibitem{Mu12}
I. Musson, Lie superalgebras and enveloping algebras. Grad. Stud. Math., vol. 131, Amer. Math. Soc., Providence, RI, 2012. \href{https://doi.org/10.1090/gsm/131}{https://doi.org/10.1090/gsm/131}. 

\bibitem{Na91}
M. Nazarov, Quantum Berezinian and the classical Capelli identity, Lett. Math. Phys. 21 (2) (1991) 123 -- 131. \href{https://doi.org/10.1007/BF00401646}{https://doi.org/10.1007/BF00401646}. 

\bibitem{Na20}
M. Nazarov, 2020. Yangian of the general linear Lie superalgebra. SIGMA 16, 112. \href{https://doi.org/10.3842/SIGMA.2020.112}{https://doi.org/10.3842/SIGMA.2020.112}. 

\bibitem{Se11}
V. Serganova, Kac-Moody superalgebras and integrability, in: Developments and Trends in Infinite
Dimensional Lie Theory, in: Progr. Math., vol. 288, Birkhäuser Boston, Inc., Boston, MA, 2011, pp. 169 -- 218. 

\bibitem{St94}
V.A. Stukopin, Yangians of Lie superalgebras of type $A(m,n)$, Funct. Anal. Appl. 28 (3) (1994) 217 -- 219. \href{https://doi.org/10.1007/BF01078460}{https://doi.org/10.1007/BF01078460}. 

\bibitem{VdL89}
J.W. Van De Leur, A classification of contragredient Lie superalgebras of finite growth, Comm. Algebra 17 (8)(1989) 1815 -- 1841. \href{https://doi.org/10.1080/00927878908823823}{https://doi.org/10.1080/00927878908823823}. 

\bibitem{Ya94-2}
H. Yamane, A Serre type theorem for affine Lie superalgebras and their quantized enveloping superalgebras, Proc. Japan Acad. Ser. A Math. Sci. 70 (1) (1994) 31 -- 36. \href{https://doi.org/10.3792/pjaa.70.31}{https://doi.org/10.3792/pjaa.70.31}. 

\bibitem{Ya99}
H. Yamane, On defining relations of affine Lie superalgebras and affine quantized universal enveloping
superalgebras, Publ. RIMS Kyoto Univ. 35 (3) (1999) 321 -- 390. \href{https://doi.org/10.2977/PRIMS/1195143607}{https://doi.org/10.2977/PRIMS/1195143607}. 

\bibitem{ZGB91}
R.B. Zhang, M.D. Gould, A.J. Bracken, Solutions of the graded classical Yang-Baxter equation and integrable models, J. Phys. A 24 (6) (1991), 1185 -- 1197. \href{https://doi.org/10.1088/0305-4470/24/6/012}{https://doi.org/10.1088/0305-4470/24/6/012}. 

\end{thebibliography}
\end{document}